\documentclass[11pt,a4paper]{amsart}

\usepackage[a4paper,margin=3cm,footskip=.5cm]{geometry}

\usepackage{yfonts}
\usepackage{hyperref}
\usepackage{amsmath}
\usepackage{amssymb}
\usepackage{amsthm}
\usepackage{mathtools}
\usepackage{enumerate}
\usepackage{cite}
\usepackage{caption}
\usepackage{subcaption}
\usepackage{graphicx}
\usepackage{bm}
\usepackage{bbm}
\usepackage{float}
\usepackage{multirow}
\usepackage{makecell}
\usepackage[table]{xcolor}
\usepackage{xcolor}
\usepackage{arydshln} 
\usepackage{boldline}
\usepackage{tabu}
\usepackage{array}
\usepackage{booktabs}
\usepackage{hhline}
\usepackage{epstopdf}
\usepackage[pagewise]{lineno}
\usepackage{color}
\usepackage{amsaddr}  
\usepackage{algpseudocode} 
\usepackage{algorithmicx} 
\usepackage{algorithm}
\usepackage{mathrsfs} 

\newcolumntype{C}[1]{>{\centering\arraybackslash}m{#1}}

\newcommand{\EQ}{\begin{equation}}
\newcommand{\EN}{\end{equation}}
\newcommand{\EQS}{\begin{equation*}}
\newcommand{\ENS}{\end{equation*}}

\newcommand{\bea}{\bed\begin{array}{rl}}
\newcommand{\eea}{\end{array}\eed}

\newcommand{\rr}{{\hbox{{\rm I}{\kern -0.22em}{\rm R}}}}

\newcommand{\Real}{\mathrm{Re}}
\newcommand{\Imag}{\mathrm{Im}}

\newcommand{\mi}{\mathrm{i}}

\theoremstyle{remark}
\newtheorem{remark}{Remark}

\theoremstyle{plain}

\theoremstyle{plain}

\def\({\left(}
\def\){\right)}
\def \R  {{\mathbb R}}
\def \C  {{\mathbb C}}

\newcolumntype{L}[1]{>{\raggedright\let\newline\\\arraybackslash\hspace{0pt}}m{#1}}
\newcolumntype{M}[1]{>{\centering\let\newline\\\arraybackslash\hspace{0pt}}m{#1}}

\newcommand{\IG}{\mathrm{IG}}

\newcommand{\GIG}{\mathrm{GIG}}
\newcommand{\BGIG}{\mathrm{BGIG}}
\newcommand{\BG}{\mathrm{BG}}
\newcommand{\VG}{\mathrm{VG}}

\usepackage{multirow}
\usepackage{array}
\usepackage{booktabs}
\newcommand{\N}{\mathbb{N}}

\newcommand{\D}{\mathrm{d}}

\newtheorem{theorem}{Theorem}[section]
\newtheorem{proposition}{Proposition}[section]

\newtheorem{lemma}[theorem]{Lemma}

\newcommand*{\green}{\textcolor{black}}
\usepackage[skip=2pt, indent=20pt]{parskip}

\usepackage{accents}
\newcommand{\ubar}[1]{\underaccent{\bar}{#1}}

\title[The bilateral generalized inverse Gaussian Process]{The bilateral generalized inverse Gaussian process with applications to financial modeling}

\author{Gaetano Agazzotti}
\address{CEREMADE, CNRS, Université Paris Dauphine-PSL}
\email{agazzotti@ceremade.dauphine.fr}

\author{Jean-Philippe Aguilar}
\address{Corresponding Author, Soci\'{e}t\'{e} G\'{e}n\'{e}rale}
\email{jean-philippe.aguilar@sgcib.com}
\date{This version: \today}
\begin{document}




\begin{abstract}
We introduce and document a class of probability distributions, called bilateral generalized inverse Gaussian (BGIG) distributions, that are obtained by convolution of two generalized inverse Gaussian distributions supported by the positive and negative semi-axis. We prove several results regarding their analyticity, shapes and asymptotics, and we introduce the associated L\'evy processes as well as their main properties. We study the behaviour of these processes under change of measure, their simulations and the structure of their sample paths, and we introduce a stock market model constructed by means of exponential BGIG processes. Based on real market data, we show that this model is easy to calibrate thanks notably to idiosyncratic properties of BGIG distributions, and that it is well suited to Monte Carlo and Fourier option pricing.
\end{abstract}

\maketitle
\section{Introduction}
\label{sec:intro}

The rich interplay between mathematics, physics and finance has given birth to a vast collection of stochastic processes whose study remains an active field of research across academia and industry. In mathematical finance in particular, two classes of processes appear to be of peculiar importance: first, the tempered stable (TS) processes, that generalize the so-called truncated L\'evy flights first introduced in \cite{Mantegna94,Kobol-Koponen} in the fully symmetric case and then extended to single sided and asymmetric processes (notably via KoBoL distributions, see \cite{BoyarchenkoIJTAF2000,Kobol-Levandorskii}), and to even more general configurations in \cite{KuchlerTS}. TS processes are particularly important because they recover (or are limiting cases) of many popular processes previously introduced in  finance to model asset returns, such as the Carr-Geman-Madan-Yor (CGMY) process introduced in \cite{CGMY} and sometimes called classical tempered stable process, or the bilateral Gamma (see \cite{kuchlerBG,KT-shapes-BG,BGPricing}) and Variance Gamma (see \cite{Madan98}) processes. Being L\'evy processes, they allow to extend the classical Black-Scholes-Merton framework of \cite{Black73} by capturing the heavy tails that are ubiquitous in financial data, while preserving the stationarity of  returns.

The second important class is arguably the class of generalized hyperbolic (GH) processes; they were introduced in \cite{BarndorffGH77} to formalize mathematically the transport of sand by wind, following the initial observations by Bagnold in the 1940s regarding the formation of dunes in the Libyan desert (paving the way to what is now known as Aeolian processes). GH processes are typically constructed as mean-variance mixtures of normal distributions, the mixing distribution being the generalized inverse Gaussian (GIG) distribution (see \cite{Eberlein04,barndorff1977} for GH distributions, and \cite{Eberlein04,barndorff2001lévy,jorgensen2012statistical} for GIG distributions), and, interestingly, have also found applications in financial modeling and risk management, as discussed for instance in \cite{Eberlein02}. They recover other processes that are well known in quantitative finance, such as the normal inverse Gaussian process (see \cite{Barndorff97}) but also the Variance Gamma processes, which demonstrates that the intersection of TS and GH classes is actually nonempty. 

In this paper, we would like to follow a somewhat different approach and, instead of considering GIG distributions only from a mixing distribution perspective, to construct a bilateral distribution by convolution of two independent GIG distributions defined on the positive and negative semi-axis (note that a similar path had already been taken notably by K\"uchler and Tappe when introducing bilateral Gamma processes in \cite{kuchlerBG}, but the generating variables were Gamma distributed in this case). The resulting distributions we obtain are double sided, and we name them bilateral generalized inverse Gaussian (BGIG) distributions. They have six parameters, like TS distributions (GH have five) and, indeed, have a nonempty intersection with the TS class; but BGIG distributions also form a very rich set on their own, giving birth to a wide range of random variables that are not TS distributed. Moreover they are infinitely divisible, allowing us to construct associated L\'evy processes; even if BGIG processes are not convolution-closed (because GIG processes themselves are not, see e.g. \cite{ZHANG2022114275}), they have a tractable characteristic function, which makes them ideal for financial applications.

The present article has therefore multiple purposes. First, we wish to define and document BGIG distributions as well as their immediate properties in Section \ref{sec:BGIG} and to study their particular and limiting cases in Section \ref{sec:particular-cases}. Thanks to recent results obtained in \cite{baricz,freitas,Godsill} on a specific family of integrals (namely Jaeger integrals, originally introduced in \cite{Jaeger42} and that happen to be involved in the L\'evy measure of BGIG distributions), we will be able to prove several results regarding the smoothness, shapes and asymptotics of BGIG distributions in Section \ref{sec:shape}. Then, we will define the BGIG L\'evy process and provide a detailed study of its behavior, paths and simulations in Section \ref{sec:BGIG-process}, allowing us to properly define a stock market model based on the exponential BGIG process in Section \ref{sec:exp-bgig}. As an application, we will calibrate this model over real market data and price European style options (via Monte Carlo and Fourier methods) in Section \ref{sec:cal-num-examples}; {\color{black} for the reader’s convenience, the {\ttfamily Python} source code for model calibration and simulation is made publicly available\footnote{\url{https://github.com/gagazzotti/Bilateral-GIG}}}. Last, Section \ref{sec:conclusion} will be dedicated to conclusions.

\section{Bilateral generalized inverse Gaussian distributions}\label{sec:BGIG}

We start our study by recalling some properties of GIG distributions, before introducing BGIG distributions as well as some of their first properties (representation of densities, cumulants).

\subsection{Notations}

In all of the following, we will consider a probability space $(\Omega,\mathcal{F},\mathbb{P})$ and, unless otherwise mentioned, expectations will be taken under $\mathbb P$, that is $\mathbb E [.] = \mathbb E ^{\mathbb P} [.]$. Given a random variable $X:\Omega \rightarrow \mathbb R$, {\color{black}we define its moment-generating function as 
$
\Phi_X (\ell) :=\mathbb E \left[ e^{\ell X} \right]
$
for all $\ell \in \R$ for which the expectation exists.
}

\subsection{Generalized inverse Gaussian distributions}
Let us recall some essential facts about the generalized inverse Gaussian (GIG) distribution; details and proofs can be found in \cite{barndorff1977,Barndorff01,Eberlein04}.

\subsubsection{Density}
{\color{black}Recall that the inverse Gaussian $\text{IG}(\lambda,\mu)$ distribution is a probability distribution whose density function $f_{\text{IG}}$ is supported by the positive semi-axis and is defined by $f_{\text{IG}}(x) := \sqrt{\lambda/(2\pi x^3)}\exp\left(-\lambda(x-\mu)^2/(2\mu^2x)\right)$ where $\lambda$ and $\mu$ are two positive real numbers.}
The GIG distribution is a three-parameter probability distribution whose density function is defined for all $x\in\mathbb R$ by :
\begin{equation}\label{GIG+_density}
    f_{\mathrm{GIG_+}} (x)
    :=
    \frac{(a/b)^{p/2}}{2K_p\left( \sqrt{ab} \right)}
    x^{p-1}
    e^{-\left( \frac{ax + b/x}{2} \right)}
    \mathbbm 1_{\{x>0\}} (x)
    ,\quad
    a,b>0, 
    \quad
    p \in\mathbb R
    ,
\end{equation}
where $K_p(.)$ denotes the modified Bessel function of the second kind (see \cite{Abramowitz72}). We will denote this distribution as $\mathrm {GIG}_+ (a,b,p)$.{\color{black}~ It is immediate to see that the $\GIG_+(a,b,p)$ distribution recovers the $\IG(\lambda,\mu)$ distribution by setting $a=\lambda/\mu^2$, $b=\lambda$ and $p=-1/2$}. It can easily be translated to the negative semi-axis to define the ``negative-sided" GIG distribution, that we denote by $\mathrm {GIG}_- (a,b,p)$, and whose density is:
\begin{equation}\label{GIG-_density}
    f_{\mathrm{GIG_-}} (x)
    :=
    \frac{(a/b)^{p/2}}{2K_p\left( \sqrt{ab} \right)}
    |x|^{p-1}
    e^{-\left( \frac{a|x| + b/|x|}{2} \right)}
    \mathbbm 1_{\{x<0\}} (x)
\end{equation}
under the same conditions on the parameters.
Let us now provide some details (Characteristic function, L\'evy measure) for the $\mathrm{GIG}_+(a,b,p)$ distribution only, the extension to the $\mathrm{GIG}_-(a,b,p)$ case being straightforward.

\subsubsection{\textcolor{black}{Moment-generating and characteristic function}}
{\color{black} The moment-generating function $M_{\GIG_+}$ of a random variable $X \sim \mathrm{GIG}_+(a,b,p)$ exists and is finite for all $\ell\in \left(-\infty, a/2\right)$, and is known under closed form (see \cite{Eberlein02}):
\begin{equation}\label{GIG+_char}
    \forall \ell\in \left(-\infty, \frac{a}{2}\right),\quad
    M_{\GIG_+} (\ell)
    = 
    \left(\frac{a}{a- 2\ell}\right)^{p/2}
    \frac{K_{p}\left(\sqrt{b(a- 2 \ell)}\right)}{K_{p}\(\sqrt{ab}\)}
    .
\end{equation}
It can be an analytically continued to the complex strip $\{z\in\C|~ \Real[z] < a/2 \}$ (see \cite[p.190]{gut2006probability}), allowing to define the corresponding characteristic function $\Phi_{\mathrm{GIG}_+} (u) := M_{\GIG_+}(\mi u) $  for $u\in\{z\in\C|~\Imag[z]>-a/2)\}$.
}


\subsubsection{L\'evy measure}
Let us introduce the Hankel function of the first kind $H_\lambda (z) = J_\lambda(z) + \mi Y_\lambda (z)$ where $J_\lambda$ (resp. $Y_\lambda$) denotes the Bessel function of the first (resp. second) kind (see \cite{Abramowitz72} or any monograph on special functions), as well as the Jaeger integral (see \cite{Jaeger42,Luke14}):
\begin{equation}
\label{eq:jaeger-definition}
    \forall x\in (0,+\infty),\quad \mathscr{I}(x , p) := \frac{2}{\pi^2}\int_0^\infty\frac{e^{-xy^2}}{y  \left|H_{|p|}(y)\right|^2}\D y.
\end{equation}
Then, defining the L\'evy measure as
\begin{multline}\label{GIG+_LevyMesure}
    \forall x\in\R,\quad \pi_{\mathrm{GIG}_+}(\D x) 
    = 
    \frac{e^{-ax/2}}{x}\left(\int_0^\infty\frac{e^{-xy}}{\pi^2 y  \left|H_{|p|}(\sqrt{2b y })\right|^2}\D y + \max(p,0)\right)\\
    \times\mathbbm{1}_{(0,\infty)}(x)\D x
\end{multline}
leads to the L\'evy-Khintchine representation for the characteristic function \eqref{GIG+_char}:
\begin{equation}
        \forall u\in\R,\quad 
        \Phi_{\mathrm{GIG}_+}(u) = \exp\left(\int_\R \left( e^{\mi u x}-1   \right) \pi_{\mathrm{GIG}_+}(\D x)\right).
\end{equation}
As noted in \cite{Godsill}, the change of measure $\tilde{y} = \sqrt{2 b y} $ allows to obtain an equivalent representation for the L\'evy measure in terms of the Jaeger integral:
\begin{equation}\label{eq:Levy_GIG}
    \pi_{\mathrm{GIG}_+}(\D x) 
    = 
    \frac{e^{-ax/2}}{x}\left(\mathscr{I}(x/2b , p) + \max(p,0)\right)\mathbbm{1}_{(0,\infty)}(x)\D x
    .
\end{equation}
The following lemma regarding asymptotic behaviours of the Jaeger integral will be useful in many occurrences throughout the paper.
\begin{lemma}
    \label{lemma:jaeger-bounds}
    Let $p\in [0,+\infty)$, we have the following properties:
    \begin{enumerate}
        \item Asymptotic behaviour in $0$:
            \begin{equation}
                \label{eq:jaeger-asympt-in-0}
                \quad \mathscr{I}(x,p) \underset{x\to 0}{=} \frac{1}{2\pi^{1/2}\sqrt{x}} + \frac{1-2p}{4} + \mathcal{O}\left(\sqrt{x}\right),
            \end{equation}
        \item  Asymptotic behaviour in $+\infty$:
        \begin{equation}
        \label{eq:jaeger-asympt-in-inf}
        \mathscr{I}(x,p) 
        \underset{x\to \infty}{=}
        \frac{x^{-|p|}}{2^{2|p|}\Gamma(|p|)} +o\left(x^{-|p|}\right).
    \end{equation}
    \end{enumerate}
\end{lemma}

\begin{proof}
    {
    \color{black}
    {For $p=0$,  \cite[Theorem D]{freitas} and \cite[Theorem I]{freitas}  prove respectively the asymptotic \eqref{eq:jaeger-asympt-in-0} and \eqref{eq:jaeger-asympt-in-inf}. For $p>0$, \eqref{eq:jaeger-asympt-in-0} and \eqref{eq:jaeger-asympt-in-inf} are given by \cite[Theorem I]{freitas}.}
    }
\end{proof}

\subsection{Bilateral generalized inverse Gaussian distributions}

We define the bilateral generalized inverse Gaussian (BGIG) distribution as the convolution
\begin{equation}
    \label{def:bgig-def}
    \BGIG(a_+,b_+, p_+,a_-,b_-,p_-) := \GIG_+(a_+,b_+,p_+)\star \GIG_-(a_-,b_-,p_-)
    ,
\end{equation}
where $\star$ stands for the convolution operator and $\GIG_+(a_+,b_+,p_+)$
(resp. $\GIG_- (a_-,b_-,p_-)$)
is the $\GIG$ distribution defined on the positive (resp. negative) semi-axis. The parameters satisfy $(a_+,b_+,a_-,b_- )\in (0,+\infty)^4$ and $(p_+,p_-)\in\R^2$. 
An example of a such distribution can be found in Figure \ref{fig:intro-dist}.
\begin{remark}\label{rem:BGIB_diff}
   Note that by construction, if $X_+\sim \GIG_+(a_+,b_+,p_+)$ and $X_-\sim \GIG_+(a_-,b_-,p_-)$, the difference is BGIG distributed, i.e. $X_+-X_-\sim \BGIG(a_+,b_+, p_+,a_-,b_-,p_-)$. 
\end{remark}
Let us give an immediate property of BGIG distributions (that will allow us to construct the associated BGIG process in Section \ref{sec:BGIG-process}):
\begin{proposition}
\label{prop:inf-div}
    BGIG distributions are infinitely divisible. 
\end{proposition}
\begin{proof}
    GIG distributions are infinitely divisible (see \cite{barndorff1977}). From Definition \ref{def:bgig-def}, it follows that BGIG distributions are infinitely divisible as well. 
\end{proof}

\subsubsection{Moment-generating and characteristic function}

 An immediate consequence of Definition \ref{def:bgig-def} and of \eqref{GIG+_char} is that the \textcolor{black}{moment-generating function} of a random variable $X \sim \BGIG(a_+,b_+, p_+,\\a_-,b_-,p_-)$ can be expressed as:
\begin{multline}
\label{eq:bgig-chf}
\color{black}
    \forall \ell\in \left(-\frac{a_-}{2},\frac{a_+}{2}\right),~M(\ell) = 
    \left(\frac{a_+}{a_+- 2\ell}\right)^{p_+/2}
    \frac{K_{p_+}\left(\sqrt{b_+(a_+- 2 \ell)}\right)}{K_{p_+}\(\sqrt{a_+b_+}\)}\\
    \color{black}
    \times
    \left(\frac{a_-}{a_- + 2\ell}\right)^{p_-/2}\frac{K_{p_-}\(\sqrt{b_-(a_- + 2 \ell)}\)}{K_{p_-}\(\sqrt{a_-b_-}\)}
\end{multline}
\textcolor{black}{and admits an analytic continuation to the strip $\{z\in\C|~\Real[z]\in(-a_-/2,a_+/2)\}$. We define the corresponding characteristic function $\Phi(u) := M(\mi u)$ for $u\in\mathscr{S}_{\BGIG} := \{z\in\C|~\Imag[z]\in (-a_+/2, a_-/2)\}$.}

\subsubsection{L\'evy measure}

From Definition \eqref{def:bgig-def} and from \textcolor{black}{\eqref{eq:Levy_GIG}}, we obtain the Lévy measure of the BGIG distribution:
\begin{multline}
    \label{eq:levy-meaure-BGIG}
    \forall x\in\R,\quad
    \pi(\D x) = 
    \left(
    \frac{e^{-a_+x/2}}{x}\left(\mathscr{I}(x/2b_+ , p_+) + \max(p_+,0)\right)\mathbbm{1}_{(0,\infty)}(x)    
    \right.
    \\
    \left.
    +  \frac{e^{-a_-|x|/2}}{|x|}\left(\mathscr{I}(|x|/2b_- , p_-)+ \max(p_-,0)\right)\mathbbm{1}_{(-\infty,0)}(x)
    \right)
    \D x.
\end{multline}
The BGIG characteristic function \eqref{eq:bgig-chf} can therefore be written in the L\'evy-Khintchine form: 
\begin{equation}\label{eq:LK}
        \forall u\in\R,\quad \Phi(u) = \exp\left(\int_\R \left( e^{\mi u x}-1   \right) \pi(\D x)\right).
\end{equation}


\subsubsection{Cumulants}
\label{subsubsec:cumul}

From Definition \ref{def:bgig-def}, the cumulants $\kappa_n$ of the $\BGIG(a_+,b_+, p_+,a_-,b_-,p_-)$ distribution can be expressed as the sum of the cumulants of $\GIG_+(a_+,b_+, p_+)$ and $\GIG_-(a_-,b_-,p_-)$. From the characteristic function \eqref{eq:bgig-chf}, it is immediate to see that the cumulants of $\GIG_-(a_-,b_-,p_-)$ are equal to the cumulants of $\GIG_+(a_-,b_-,p_-)$ for even order and opposite otherwise.
If we denote $\kappa_n^+$ (resp. $\kappa_n^-$) the cumulants of $\GIG_+(a_+,b_+, p_+)$ (resp. $\GIG_+(a_-,b_-,p_-)$), we can therefore write:
\begin{equation}
    \kappa_n = \kappa_n^+ + (-1)^n\kappa_n^-.
\end{equation}
The expression of $\kappa_n^+$ and $\kappa_n^-$ can be found up to order four in \cite[Eq. 2.25]{jorgensen2012statistical}; we therefore have:
\begin{equation}
    \kappa_n^\pm = W_{n}(\omega_\pm)\eta_\pm^n
\end{equation}
where $\omega_\pm := \sqrt{a_\pm b_\pm}$, $\eta:= \sqrt{a_\pm/b_\pm}$ and:
\begin{equation}\label{W_k}
\left\{
    \begin{aligned}
        W_1(\omega_\pm) &= R(\omega_\pm),\\
        W_2(\omega_\pm) &= -R(\omega_\pm)^2+\frac{2(p_\pm+1)}{\omega_\pm}R(\omega_\pm)+1,\\
        W_3(\omega_\pm) &= 2R(\omega_\pm)^3-\frac{6(p_\pm+1)}{\omega_\pm}R(\omega_\pm)^2+\left(\frac{4(p_\pm+1)(p_\pm+2)}{\omega_\pm^2}-2\right)R(\omega_\pm)\\
        &\quad+\frac{2(p_\pm+2)}{\omega_\pm},\\
        W_4(\omega_\pm) &= -6R(\omega_\pm)^4 + \frac{24(p_\pm+1)}{\omega_\pm} + \left(8-\frac{4(p_\pm+1)(7p_\pm+11)}{\omega_\pm^2}\right)R(\omega_\pm)^2\\
        & \quad+\left(\frac{8(p_\pm+1)(p_\pm+2)(p_\pm+3)}{\omega_\pm^3}- \frac{4(4p_\pm+5)}{\omega_\pm}\right)R(\omega_\pm) \\
        &\quad+ \frac{4(p_\pm+1)(p_\pm+2)}{\omega_\pm^2}-2.
    \end{aligned}
\right.
\end{equation}
In \eqref{W_k}, $R(\omega_\pm) = K_{p_\pm+1}(\omega_\pm)/K_{p_\pm}(\omega_\pm)$, where the function $K_\nu$ stands for the modified Bessel function of second kind (MacDonald function) \cite{Abramowitz72}. 

\begin{proposition}
    Let $X \sim \BGIG(a_+,b_+, p_+,a_-,b_-,p_-)$, we have:
    \begin{multline}
        \forall n\in\N,\quad\mathbb{E}\left[X^n\right] = \frac{1}{K_{p_+ }\left(\sqrt{a_+b_+}\right)K_{p_- }\left(\sqrt{a_-b_-}\right)}\\
        \times \sum_{k=0}^n\binom{n}{k} (-1)^{n-k} K_{p_+ +k }\left(\sqrt{a_+b_+}\right)K_{p_- + n-k }\left(\sqrt{a_-b_-}\right).
    \end{multline}
\end{proposition}

\begin{proof}
    Using Definition \ref{def:bgig-def} and the expression of the moment for GIG distributions (see \cite{jorgensen2012statistical}), the proof is straightforward.
\end{proof}

\subsubsection{Mellin transform}


In this subsection, we provide a representation of the BGIG density in the Mellin space, that can be used as an alternative to the representation in the Fourier space \eqref{eq:bgig-chf}. 
{\color{black}Recall that for a function $f$ defined over the positive reals, the Mellin transform $f^*$ of $f$ is defined by $f^*(s) := \int_0^{+\infty}x^{s-1}f(s)\D s$ for any $s\in\C$ such that the integral is well defined (see \cite[Eq. 1]{Flajolet95}).}  We start by a lemma on GIG densities:

\begin{lemma}\label{lemma:GIG_Mellin}
    For every $x> 0$, the GIG density admits the representation:
    \begin{equation}\label{eq:GIG_Mellin}
        f_{\mathrm{GIG}_+}(x)
        =
        \int\limits_{c-\mi\infty}^{c+\mi\infty}
        \frac{K_{s+p-1}(\sqrt{ab})}{K_p(\sqrt{ab})}
        \left( \frac{a}{b} \right)^{\frac{1-s}{2}}
        x^{-s}
        \frac{\D s}{2\mi \pi}
        ,
        \quad
        c\in\mathbb R
        .
    \end{equation}
\end{lemma}
\begin{proof}
    By Definition of the Mellin transform, the Mellin transform of the GIG density is:
    \begin{equation}
        f_{\mathrm{GIG}_+}^*(s)
         =
        \frac{(a/b)^{p/2}}{2K_p\left( \sqrt{ab} \right)}
        \int_0^\infty
        x^{s+p-2}
        e^{-\left( \frac{ax + b/x}{2} \right)}
        \D x = \left( \frac{a}{b} \right)^{\frac{1-s}{2}}
        \frac{K_{s+p-1}(\sqrt{ab})}{K_p(\sqrt{ab})}
    \end{equation}
    where the integral converges for all $s\in\C$. Applying the Mellin inversion theorem \cite[Eq. 7]{Flajolet95} yields \eqref{eq:GIG_Mellin}. 
\end{proof}

\begin{proposition}\label{prop:BGIG_Mellin}
   For every $x> 0$, the BGIG density admits the representation:
    \begin{multline}\label{eq:BGIG_Mellin}
        f(x)
        = \sqrt{\frac{a_+a_-}{b_+b_-}}
        \int\limits_{c+\mi\mathbb R^2}
       \frac{\Gamma(s_1+s_2-1)\Gamma(1-s_2)}{\Gamma(s_1)}\\
       \times
        \frac{K_{ s_1 +  p_+ - 1}(\sqrt{a_+b_+})
        K_{ s_2 +  p_- - 1}(\sqrt{a_-b_-})}{K_{p_+}(\sqrt{a_+b_+})K_{p_-}(\sqrt{a_-b_-})}\\
        \times
        \left( \frac{a_+}{b_+} \right)^{-\frac{s_1}{2}}
        \left( \frac{a_-}{b_-} \right)^{-\frac{s_2}{2}}
        x^{1-s_1-s_2}
        \frac{\D s_1 \D s_2}{(2\mi\pi)^2}
        ,
    \end{multline}
    where $c$ belongs to the polyhedron $P = \{ (s_1,s_2) \in\C^2, \Real (s_1)>0, \Real (s_2)<1, \Real (s_1+s_2-1)>0 \}$.
\end{proposition}
\begin{proof}
    From Definition \eqref{def:bgig-def}, the BGIG density can be written as the convolution of two GIG densities as follows:
    \begin{equation}
        f(x)
        = 
        \int_{-\infty}^\infty f_{\mathrm{GIG}_+}(y,a_+,b_+,p_+) f_{\mathrm{GIG}_-}(x-y,a_-,b_-,p_-) 
        \D y
        ,
    \end{equation}
    where we have used a slight abuse of notations to indicate the parameters involved in each density function. Using Definitions \eqref{GIG+_density} and \eqref{GIG-_density} we are left with, for $x> 0$:
    \begin{equation}
        f(x)
        = 
        \int\limits_0^\infty f_{\mathrm{GIG}_+}(x+y, a_+,b_+,p_+) f_{\mathrm{GIG}_+}(y,a_-,b_-,p_-)
        \D y
        .
    \end{equation}
    From Lemma \ref{lemma:GIG_Mellin}, we can therefore write:
    \begin{multline}
        f(x) = \int\limits_{c_1-\mi\infty}^{c_1+\mi\infty} \int\limits_{c_2-\mi\infty}^{c_2 +\mi\infty}
        \frac{K_{s_1+p_+-1}(\sqrt{a_+b_+})K_{s_2+p_- -1}(\sqrt{a_-b_-})}{K_{p_+}(\sqrt{a_+b_+})K_{p_-}(\sqrt{a_-b_-})}
        \left( \frac{a_+}{b_+} \right)^{\frac{1-s_1}{2}}
        \left( \frac{a_-}{b_-} \right)^{\frac{1-s_2}{2}}
        \\
        \times
        \int_0^\infty (x+y)^{-s_1} y^{-s_2} \D y
        \frac{\D s_1 \D s_2}{(2\mi \pi)^2}
    \end{multline}
    where the $y$ integral is a particular case of a Beta integral:
    \begin{equation}
        \int_0^\infty (x+y)^{-s_1} y^{-s_2} \D y
        =
        \frac{\Gamma(s_1+s_2-1)\Gamma(1-s_2)}{\Gamma(s_1)}
        x^{1-s_1-s_2}
        ,
    \end{equation}
    which converges if $\Real (s_1)>0$, $\Real (1-s_2)>0$ and $\Real (-1+s_1+s_2)>0$. Applying the Mellin inversion theorem yields \eqref{eq:BGIG_Mellin}.
\end{proof}
\begin{remark}
    It is straightforward to extend Proposition \ref{prop:BGIG_Mellin} to the case $x<0$, as it follows from the characteristic function \eqref{eq:bgig-chf} that under the transformation $x\rightarrow -x$ the parameters are transformed as $(a_+,b_+,p_+,a_-,b_-,p_-) \rightarrow (a_-,b_-,p_-,a_+,b_+,p_+)$.
    See also Proposition \ref{prop:symetry} in Section \ref{sec:shape}.
\end{remark}

\section{Particular cases}

\label{sec:particular-cases}
In this section, we study some particular cases of BGIG distributions, as well as some of their limiting cases. 

\subsection{Decomposition of BGIG distributions}
\label{subsec:conv-with-BG}
Let $D$ be the distribution defined by its characteristic function:
\begin{equation}
        \forall u\in\R,\quad \Phi_\mathrm{D}(u) = \frac{K_{p_+}(\sqrt{b_+(a_+-\mi 2 u)})K_{p_-}(\sqrt{b_-(a_-+\mi 2 u)})}{K_{p_+}(\sqrt{a_+b_+})K_{p_-}(\sqrt{a_-b_-})}
        .
\end{equation}
From \eqref{eq:levy-meaure-BGIG}, it is immediate to see that any BGIG distribution with parameters $p_\pm>0$ can be written as
    \begin{multline}
        \BGIG(a_+,b_+, p_+,a_-,b_-,p_-) = \mathrm{D}(a_+,b_+, p_+,a_-,b_-,p_-)\star \BG(p_+/2,a_+/2,p_-/2,a_-/2)
        ,
    \end{multline}
where $\BG(\alpha_+,\lambda_+,\alpha_-,\lambda_-)$ denotes the bilateral Gamma (BG) distribution whose parameters are strictly positive and whose characteristic function is (see details in \cite{kuchlerBG,KT-shapes-BG,aguilarkirkby2022robust})
\begin{equation}
    \Phi_{\mathrm{BG}} (u)
    =
    \left(\frac{\lambda_+}{\lambda_+-\mi u}\right)^{\alpha_+}
    \left(\frac{\lambda_-}{\lambda_-+\mi u}\right)^{\alpha_-}
\end{equation}
for all $u\in\mathbb R$.

\begin{remark}
    Let us notice that, even if they can be represented by a convolution with a BG density (in the case $p_\pm>0$), BGIG distributions are not a subclass of tempered stable (TS) distributions \cite{KuchlerTS}; indeed from
    Lemma \ref{lemma:jaeger-bounds} and Definition \ref{eq:levy-meaure-BGIG}, it is immediate to see that for $p_\pm\leq0$ the BGIG L\'evy measure has the following behaviours:
    \begin{equation}\label{eq:pi_behav}
    \left\{
        \begin{aligned}
            \pi(\D x) &= \mathcal{O}\left(\frac{e^{-a_\pm|x|/2}}{|x|^{1+|p_\pm|}}\right)\D x 
            \quad\mathrm{as}\quad
            x\to\pm\infty, \\
            \pi(\D x) &= \mathcal{O}\left(x^{-3/2}\right)\D x
            \quad\mathrm{as}\quad
            x\to 0.
        \end{aligned}
    \right.
    \end{equation}
    Such difference between asymptotics in $\pm\infty$ and 0 is impossible to obtain with TS distributions whose L\'evy measure is given by
    \begin{equation}\label{eq:Levy_measure_TS}
        \forall x\in\R,\quad \pi_{\mathrm{TS}}(\D x) = \left( \frac{\alpha_+}{x^{1+\beta_+}}e^{-\lambda_+ x}\mathbbm{1}_{(0,+\infty)}(x)+\frac{\alpha_-}{|x|^{1+\beta_-}}e^{-\lambda_- |x|}\mathbbm{1}_{(-\infty,0)}(x)\right)\D x
    \end{equation}
    where  $\alpha_\pm,\lambda_\pm > 0$ (like in the BG case) and $\beta_\pm \in (0,1)$. The only exception occurs for $p_\pm= - 1/2$: in this case both behaviours in $\eqref{eq:pi_behav}$ coincide and the BGIG distribution degenerates into a bilateral inverse Gaussian distribution (see Section \ref{subsec:half}), which is a TS distribution with $\beta_\pm= 1/2$. For $p_\pm>0$, the closure of BGIG distributions recovers other TS distributions, as we will see in the next sections.
\end{remark}

\subsection{Bilateral Gamma and Gamma distributions}
For $p_\pm > 0$, we have the following convergence in law:
\begin{equation}
    \BGIG\left(a_+,b_+,p_+,a_-,b_-,p_-\right)\overset{\mathcal{L}}{\underset{b_+,b_- \to 0}{\longrightarrow}} \BG(p_+,a_+/2, p_-,a_-/2)
    .
\end{equation}
This is an immediate consequence of the weak convergence of each of the GIG subordinators to Gamma subordinators (see \cite[Eq. 3.3]{kuchlerBG}). Letting $p_-\to 0$ recovers the usual Gamma distribution, which is supported by the positive semi-axis only.
\begin{remark}
BG distributions belong then to the closure of BGIG class of distributions, however, as proved in \cite{KT-shapes-BG}, these distributions are not smooth whereas BGIG distributions are (see Proposition \ref{prop:BGIG-smooth}).
\end{remark}

\subsection{Variance Gamma and symmetric Variance Gamma distributions}
If $p_\pm = p>0$, the following convergence holds:
\begin{equation}
    \BGIG\left(a_+,b_+,p,a_-,b_-,p\right)\overset{\mathcal{L}}{\underset{b_+,b_- \to 0}{\longrightarrow}} \VG(\sigma^2,\theta,\nu),
    \end{equation}
where $\VG$ denotes the Variance Gamma distribution (see \cite{Madan98,AguilarKirkby21}); it is usually parametrized via a scale (volatility) parameter $\sigma$, an asymmetry (or skewness) parameter $\theta$ and a kurtosis parameter $\nu$, and has the following characteristic function:
\begin{equation}\label{eq:VG_char}
    \Phi_{\mathrm{VG}} (u) 
    =
    \left(  1- i \theta\nu u + \frac{\sigma^2\nu}{2}u^2 \right)^{-\frac{1}{\nu}}
    .
\end{equation}
Comparing \eqref{eq:VG_char} and \eqref{eq:bgig-chf} in the limits $b\pm\to 0$ and $p_\pm =p$ allows to express the VG parameters in terms of the BGIG parameters as (see also \cite{aguilarkirkby2022robust}) $\sigma^2 =8p(a_+a_-)^{-1}$, $\theta = 2p(a_+^{-1}-a_-^{-1})$ and $\nu =p^{-1}$.
The VG distribution is itself a particular case of the CGMY distribution \cite{CGMY} with $Y=1$ (which are themselves particular cases of the more general KoBoL distributions \cite{Kobol-Levandorskii}). If furthermore $a_\pm = a$, then the BGIG distribution becomes the symmetric Variance Gamma distribution (see \cite{Madan90,Aguilar20}) for which the asymmetry parameter $\theta$ is equal to zero.

\subsection{Half-integer values of \texorpdfstring{$p$}{p}}
\label{subsec:half}
In this subsection, we focus on negative parameters $p_\pm$ only since, for $p_\pm$ positive, the distribution is simply the convolution with a bilateral Gamma distribution (see Section \ref{subsec:conv-with-BG}). For simplicity we restrict ourselves to the case $p_+ =p_- =p$. 
\subsubsection{\texorpdfstring{Case $p=-1/2$}{pf}: the ``bilateral inverse Gaussian" distribution}

For $p=-1/2$, the Jaeger integral, which is symmetric in $p$, is available in closed form (see \cite{baricz}):
\begin{equation}
    \mathscr{I}(x,1/2) = \frac{1}{2\sqrt{\pi x}}
    ,
\end{equation}
and the L\'evy measure simplifies to:
\begin{equation}
    \label{eq:levy-meaure-BGIG_-1/2}
    \forall x\in\R,\quad
    \pi(\D x) =
    \left(
    \sqrt{\frac{b_+}{2\pi x^3}}
    e^{-a_+x/2}\mathbbm{1}_{(0,\infty)}(x) + \sqrt{\frac{b_-}{2\pi |x|^3}}
    e^{-a_-|x|/2}\mathbbm{1}_{(-\infty,0)}(x)
    \right)
    \D x.
\end{equation}
We recognize the distribution of the difference of two inverse Gaussian distributions and we can write:
\begin{equation}
\BGIG(-1/2,a_+,b_+,-1/2,a_-,b_-) = \IG_+(a_+, b_+)*\IG_-(a_-, b_-). 
\end{equation}
This distribution is a special case of Tempered Stable (TS) distribution with, under the parametrization of \eqref{eq:Levy_measure_TS}, TS parameters equal to $\lambda_\pm = a_\pm/2$, 
$\beta_\pm = 1/2$ and $\alpha_\pm = \sqrt{b_\pm/2\pi}$; it corresponds to the distribution of the difference of two independent inverse Gaussian variables. 

\subsubsection{\texorpdfstring{Case $p=-3/2$}{pf}}
In general, BGIG distributions are not a particular case of TS distributions (the case $p=-1/2$ being an exception). Indeed, let us assume $p = -3/2$, in that case we have (see \cite{baricz}):
\begin{equation}
    \mathscr{I}(x,3/2) = \frac{1}{2}\left(\frac{1}{\sqrt{\pi x}}-e^x\mathrm{erfc}(\sqrt{x})\right).
\end{equation}
The corresponding Lévy measure is then:
\begin{multline}
    \label{eq:levy-meaure-3/2}
    \forall x\in\R,\quad
    \pi(\D x) = 
    \left(
    \sqrt{\frac{b_+}{2\pi x^3}}
    e^{-a_+x/2}\mathbbm{1}_{(0,\infty)}(x)
    + 
    \sqrt{\frac{b_-}{2\pi |x|^3}}
    e^{-a_-|x|/2}\mathbbm{1}_{(-\infty,0)}(x)
    \right.
    \\
    +x^{-1}\mathrm{erfc}(\sqrt{x/2b_+})
    e^{x/2b_+-a_+x/2}\mathbbm{1}_{(0,\infty)}(x)\\
    + 
    \left.
    |x|^{-1}\mathrm{erfc}(\sqrt{|x|/2b_-})
    e^{|x|/2b_--a_-|x|/2}\mathbbm{1}_{(-\infty, 0)}(x)
    \right)
    \D x.
\end{multline}
where erfc stands for the complementary error-function. The first line of \eqref{eq:levy-meaure-3/2} is the Lévy measure of a TS distribution with $\beta_\pm = 1/2$, but the second line is a more complex component showing that BGIG distributions are not a subclass of Tempered-Stable distributions. 
\subsubsection{Case \texorpdfstring{$p=-(n+1/2)$}{hn}.}
More generally, it is shown in \cite{baricz} that for $p = -(n+ 1/2)$ with $n\in\N_{>1}$, we have:
\begin{equation}
    \mathscr{I}(x,n+1/2) = \frac{1}{\pi}\int_0^\infty \frac{u^{2n}e^{-xu^2}}{\textcolor{black}{\prod_{i=1}^n(u^2+\alpha_i)}}\D u
\end{equation}
and therefore
\begin{equation}
    \mathscr{I}(x,n+1/2) = \frac{B_0}{\sqrt{x}}+\sum_{j=1}^nB_j e^{\alpha_j^2 x}\left(\frac{\alpha_j}{|\alpha_j|}-1
    +\mathrm{erfc}(\alpha_j\sqrt{x})\right)
\end{equation}
where the $B_j$ are constants and the $\alpha_j$ are the zeros of the modified Bessel function of second kind. The particular values of the Jaeger integrals generate a component involving the erfc function in the L\'evy measure, demonstrating yet again that these distributions do not belong to the TS class.
\section{Shape and asymptotics of BGIG distributions}

Let us now study the smoothness and analyticity of BGIG distributions, as well as some asymptotic behaviours for large arguments of BGIG random variables as well as their maximum and minimum.

\label{sec:shape}

\subsection{Smoothness} 
In this section, we prove that BGIG distributions are strictly unimodal and smooth, which can be seen in the example of a BGIG distribution $f$ displayed in Figure \ref{fig:intro-dist}. First, remark that $x\mapsto \mathscr{I}(x,p)$ is decreasing on $(0,+\infty)$ (which is immediate by differentiating with respect to $x$ and noticing that the integrand of \eqref{eq:jaeger-definition} is negative). As a consequence, if we re-write the BGIG L\'evy measure \eqref{eq:levy-meaure-BGIG} as:
\begin{equation}
    \forall x\in\R,\quad \pi(\D x) = \frac{k(x)}{x}\D x
\end{equation}
where $k$ is defined as:
\begin{multline}
     \forall x\in\R,\quad
    k(x) := 
    e^{-a_+x/2}\left(\mathscr{I}(x/2b_+ , p_+) + \max(p_+,0)\right)\mathbbm{1}_{(0,\infty)}(x)    
    \\
    -  e^{-a_-|x|/2}\left(\mathscr{I}(|x|/2b_- , p_-)+ \max(p_-,0)\right)\mathbbm{1}_{(-\infty,0)}(x)
    ,
\end{multline}
then it follows from Lemma \ref{lemma:jaeger-bounds}, that $x\rightarrow k(x)$ is decreasing on both $(-\infty,0)$ and $(0+\infty)$. From \cite[Corollary 15.11]{Satolévy}, this implies that BGIG distributions are self-decomposable and hence of class $L$ in the terminology of \cite{sato-class-l}. 

\begin{figure}
        \centering
        \includegraphics[scale = 0.9]{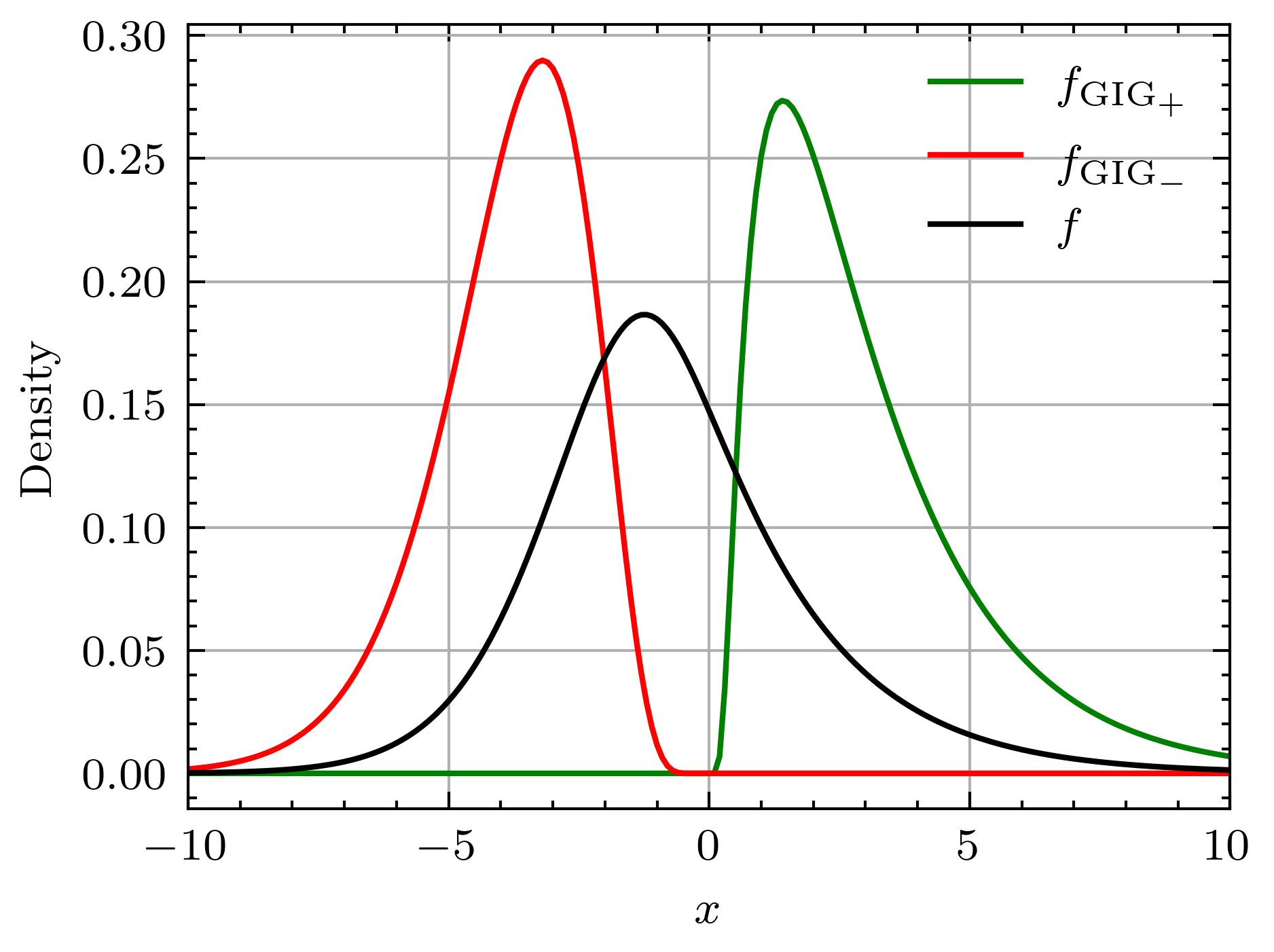}
        \caption{$\BGIG(1,2,1,3,4,5)$ distribution and its two underlying GIG distributions.}
        \label{fig:intro-dist}
\end{figure}

\begin{proposition}
    \label{prop:BGIG-smooth}
    Densities of BGIG distributions are smooth, i.e., they belong to $\mathscr{C}^\infty(\R,\R)$.
\end{proposition}

\begin{proof}
    From Lemma \ref{lemma:jaeger-bounds}, we have that $k(0^+) = |k(0^-)| = +\infty $ and thus, neither the density nor its reflection is of type $I_4$ in the sense of \cite{sato-class-l}. The use of \cite[Theorem 1.9]{sato-class-l} concludes the proof.
\end{proof}

\subsection{Analyticity}
Before discussing the analyticity of the density, let us remark that, for $a,b>0$:
\begin{equation}
\left\{
    \begin{aligned}
        \label{eq:arg-module}
        & \arg \sqrt{b(a\pm2\mi u)} \underset{u\to\infty}{\longrightarrow} \pm \frac{\pi}{4},\\
    & \left|\sqrt{b(a\pm2\mi u)}\right|  = \sqrt{b}(a^2+4u^2)^{1/4} =  \sqrt{2bu}(1+\mathcal{O}(u^{-2})).
    \end{aligned}
\right.
\end{equation}
The following Lemma will also be useful. 
\begin{lemma}\label{lem:Phi_asympt}
    If $X\sim\BGIG(a_+,b_+,p_+,a_-,b_-,p_-)$, then its characteristic function $\Phi$ satisfies:
    \begin{equation}
        \Real\left[\Phi(u)\right] \underset{u\to\infty}{\sim} Ku^{-(1+\Bar{p})/2}e^{-\sqrt{u}\Bar{b}}
        \cos\left(-\ubar{b}\sqrt{u}+\frac{\pi}{4}\ubar{p}\right),
    \end{equation}
    where:
    \begin{equation}
        \label{eq:def-asymptotics-re-phi}
            K :=\frac{\pi a_+^{p_+/2}a_-^{p_-/2}}{2^{(\Bar{p}+3)/2}
            (b_+b_-)^{1/4}
            K_{p_+}(\sqrt{a_+b_+})K_{p_-}(\sqrt{a_-b_-})}\\
    \end{equation}
    and $\Bar{p} := p_++p_-$, $\ubar{p} := p_+-p_-$, $\Bar{b} = \sqrt{b_+}+\sqrt{b_-}$, $\ubar{b} := \sqrt{b_-}-\sqrt{b_+}$.
\end{lemma}

\begin{proof}
    The BGIG characteristic function \ref{eq:bgig-chf} can be written as:
    \begin{equation}
    \Phi(u) = M
    \frac{K_{p_+}\left(\sqrt{b_+(a_+-\mi 2 u)}\right)K_{p_-}\(\sqrt{b_-(a_- +\mi 2 u)}\)}{(a_+-\mi 2u)^{p_+/2}(a_-+\mi 2u)^{p_-/2}}.
    \end{equation}
    where $M=a_+^{p_+/2}a_-^{p_-/2}/(K_{p_+}(\sqrt{a_+b_+})K_{p_-}(\sqrt{a_-b_-}))$. 
    We first note that, with notations \eqref{eq:def-asymptotics-re-phi}, Equation \ref{eq:arg-module} gives:
    \begin{equation}
    \label{eq:product-power}
        (a_+-\mi 2u)^{-p_+/2}(a_-+\mi 2u)^{-p_-/2} = (2u)^{-\Bar{p}/2}e^{\mi \frac{\pi}{4}\ubar{p}}(1+\mathcal{O}(u^{-2})).
    \end{equation}
    Since $\lim_{u\to+\infty}|\arg{ \sqrt{b(a\pm 2\mi u)}}| < 3\pi /2$, we have by \cite[Eq. 10.40.2]{NIST:DLMF}:
    \begin{equation}
      K_{p_\pm}\left(\sqrt{b_\pm(a_\pm\mp\mi 2 u)}\right) \underset{u\to\infty}{\sim}
      \sqrt{\frac{\pi}{2\sqrt{b_\pm}(a_\pm^2+4u^2)^{1/4}}e^{\pm\mi \pi/4}}e^{-\sqrt{b_\pm}(a_\pm^2+4u^2)^{1/4}e^{\mp\mi \pi/4}}
        .
    \end{equation}
    Inserting this asymptotic in \eqref{eq:arg-module}, and re-arranging the terms, we get:
    \begin{equation}
        \label{eq:bessel-prod-simplified}
        K_{p_+}\left(\sqrt{b_+(a_+-\mi 2 u)}\right)K_{p_-}\(\sqrt{b_-(a_- +\mi 2 u)}\) 
        \underset{u\to\infty}{\sim}
        \frac{\pi}{2\sqrt{2}(b_+ b_-)^{1/4} \sqrt{u}}e^{-\sqrt{u}(\Bar{b}+\mi\ubar{b})}.
    \end{equation}
    Combining \ref{eq:bessel-prod-simplified} and \ref{eq:product-power} yields:
    \begin{equation}        
        \Phi(u)
        \underset{u\to\infty}{\sim}
         Ku^{-(1+\bar{p})/2}e^{-\sqrt{u}(\Bar{b}+\mi\ubar{b})+\mi \frac{\pi}{4}\ubar{p}}
        .
    \end{equation}
    Taking the real part concludes the proof.
\end{proof}

\begin{proposition}
    Let $f$ be the density function of the $\BGIG(a_+,b_+,p_+,a_-,b_-,p_-)$-distribution. The following holds:
    \begin{enumerate}
        \item $f$ is not real-analytic in 0,
        \item $f$ is real-analytic on $\R\backslash\{0\}$.
    \end{enumerate}
\end{proposition}

\begin{proof}
    To prove these results, we use the  Fourier–Bros–Iagolnitzer (FBI) transform defined as \cite{Hormander03,krantz2002primer}:
    \begin{equation}\label{FBI_def}
        T_t[\varphi](\textcolor{black}{x},\xi) := \int_\R \varphi(s)e^{-\mi2\pi s \xi}e^{-\pi t (s-x)^2}\D s 
    \end{equation}
    where $(t,\xi,x)\in (0,\infty)\times \R^2$ and $\varphi\in L^1(\R)$. In the case of the BGIG density (i.e. taking $\varphi=f$ in \eqref{FBI_def}), we therefore have:
    \begin{equation}
            T_t[f](\textcolor{black}{x},t\xi) = e^{-\pi x^2 t}\left(\Phi\star g_t\right)(t\xi+\mi tx)
    \end{equation}
    where $g_t(u) := (\pi t)^{-1/2}\exp(-u^2/\pi t)$. This equality is obtained by expanding the square in the exponential term, using the convolution property of the Fourier transform and noticing that $\mathcal{F}[f] = \Phi$. 
    
    \begin{itemize}
    \item[$\bullet$] When $x\neq 0$, it is easy to see that $t\mapsto\Phi(t\xi+\mi tx)$ is bounded on $(0,\infty)$, it is then immediate that $\Phi\star g_t$ is also bounded. As a consequence, there exists $K>0$ such that:
    \begin{equation}
        |T_t[f](\textcolor{black}{x},t\xi)|\leq K e^{-\pi x^2 t}
    \end{equation}
    which, in virtue of \cite[Thm. 5.3.2]{krantz2002primer}, ensures analyticity of $f$ for $x\neq 0$.
    \item[$\bullet$] For $x=0$, we have:
    \begin{equation}
        T_t[f](\textcolor{black}{0},t\xi) =\left(\Phi\star g_t\right)(t\xi).
    \end{equation}
    The convolution can be approximated by the introducing a truncation centered around $t\xi$ and of size $2C\sqrt{t}$ for $C>0$:
    \begin{multline}\label{eq:trunc}
        \left(\Phi\star g_t\right)(t\xi) = \int_\R g_t(t\xi-u)\Phi(u)\D u 
        =\int_{t\xi-C\sqrt{t}}^{t\xi+C\sqrt{t}} g_t(t\xi-u)\Phi(u)\D u + \mathcal{O}\(C^{-1}e^{-C^2/\pi}\)
    \end{multline}
  where the $\mathcal{O}(C^{-1}e^{-C^2/\pi})$ terms comes from the asymptotic behavior of the complementary error function $\mathrm{erfc}$ (see \cite[Eq. 7.12.1]{NIST:DLMF}).  Applying the Mean Value theorem to the real part of the truncated integral in \ref{eq:trunc}, there exists $\eta_t \in [t\xi-C\sqrt{t},t\xi+C\sqrt{t}]$ such that:
    \begin{equation}
            \Real\left[\int_{t\xi-C\sqrt{t}}^{t\xi+C\sqrt{t}} g_t(t\xi-u)\Phi(u)\D u\right]
            = 2C\sqrt{t}g_t(t\xi-\eta_t)\Real\left[\Phi(\eta_t)\right].
    \end{equation}
    Since $\eta_t \in [t\xi-C\sqrt{t},t\xi+C\sqrt{t}]$, we have $g_t(t\xi-\eta_t) = \mathcal{O}(t^{-1/2}C^{-1}e^{-C^2/\pi t}))$.
    If we assume $\xi\neq 0$, then $\eta_t \to \infty$ when $t\to\infty$ and thus, from the asymptotic of Lemma \ref{lem:Phi_asympt}, we have:
    \begin{multline}\label{eq:decay}
        \Real\left[\int_{t\xi-C\sqrt{t}}^{t\xi+C\sqrt{t}} g_t(t\xi-u)\Phi(u)\D u\right] = 
        \mathcal{O}\left(e^{-C^2/(\pi t)}(t\xi)^{-(1+\Bar{p})/2}e^{-\sqrt{t\xi}\Bar{b}}
        \cos\left(-\ubar{b}\sqrt{t\xi}+\frac{\pi}{4}\ubar{p}\right)\right)
        .
    \end{multline}
    Noting that
    $
        \left|\left(\Phi\star g_t\right)(t\xi) \right| \geq \left|\Real\left[\left(\Phi\star g_t\right)(t\xi)\right]\right|
    $
    shows that there does not exist $M>0$ such that for $|\xi|>M$, $\left(\Phi\star g_t\right)(t\xi)$ decays as $e^{-at}$ for some positive $a$ (here \eqref{eq:decay} shows that the decay is at best in $e^{-a\sqrt{t}}$). Thus, $f$ does not satisfy the RA(0) condition of \cite[Thm. 5.3.2]{krantz2002primer} and therefore is not analytic in $x=0$. 
    \end{itemize}
\end{proof}

\subsection{Unimodality}
\begin{proposition}
    \label{prop:unimodal}
    BGIG distributions are strictly unimodal, i.e., there exists a unique $x_0\in \R$ such that their density function is strictly increasing on $(-\infty,x_0)$ 
    and strictly decreasing on $(x_0, +\infty)$.
\end{proposition}

\begin{proof}
    The same argument exposed in the proof of Proposition \ref{prop:BGIG-smooth} and the use of \cite[Theorem 1.2]{sato-class-l} instead of \cite[Theorem 1.9]{sato-class-l} concludes the proof.
\end{proof}

\begin{proposition}
    The mode $x_0$ of a BIG distribution is located on $(0,\infty)$ if and only if:
    \begin{multline}
        (a_-b_+-a_+b_-)K_{1-p_+-p_-}\left(\sqrt{\Lambda}\right)>
        \frac{2(\Lambda-(a_++a_-)(b_+p_-+b_-p_+))K_{2-p_--p_+}\left(\sqrt{\Lambda}\right)}{\sqrt{b_-+b_+}}
    \end{multline}
    where $\Lambda := (a_++a_-)(b_++b_-)$.
    The mode is located $(-\infty,0)$ if and only if the inequality is strictly reversed. Finally, $x_0 = 0$ if and only if the inequality becomes an equality.
\end{proposition}

\begin{proof}
    Using Proposition \ref{prop:BGIG-smooth} and simplifying positive multiplicative terms, we get:
    \begin{multline}
        \mathrm{sgn} [f_\BGIG '(0)] = \mathrm{sgn}\bigg[[(a_-b_+-a_+b_-)K_{1-p_+-p_-}\left(\sqrt{(a_++a_-)(b_++b_-)}\right)\\
        +\frac{2((a_++a_-)(b_+(1-p_-)+b_-(1-p_+)))K_{2-p_--p_+}\left(\sqrt{(a_++a_-)(b_++b_-)}\right)}{\sqrt{b_-+b_+}}\bigg]
    \end{multline}
    and the result follows immediately from Proposition \ref{prop:unimodal} and Proposition \ref{prop:BGIG-smooth} that ensures continuity of the first derivative. 
    \end{proof}

    \subsection{Asymptotics}
\begin{proposition}
    \label{prop:symetry}
    Let $a_\pm,b_\pm>0$, $p_\pm\in\R$, and denote by $f$ the density function of the distribution $\BGIG(a_+,b_+,p_+,a_-,b_-,p_-)$ and by $\widetilde{f}$ the density function of the distribution $\BGIG(a_-,b_-,p_-,
    a_+,b_+,p_+)$. We have the symmetry relation:
    \begin{equation}
        \forall x\in\R,\quad f(x) = \widetilde{f}(-x).
    \end{equation}
\end{proposition}

\begin{proof}
    This is an immediate consequence of the subsitution $u\rightarrow -u$ in the BGIG characteristic function \ref{eq:bgig-chf}.
\end{proof}

\begin{proposition}
    \label{prop:asymptotic}
    Let $f$ be the density function of a $\BGIG(a_+,b_+,p_+,a_-,b_-,p_-)$ distribution. We have the asymptotic behaviour:
    \begin{equation}
            \displaystyle  f(x)
            \underset{x\to\pm\infty}{\sim} Z_\pm |x|^{p_\pm-1}
            \exp \left(-\frac{a_\pm |x|}{2}\right)
    \end{equation}
    where:
    \begin{equation}
        Z_\pm := \left(
            \frac{a_\pm}{b_\pm}
            \right)^{p_\pm/2}\left(\frac{a_\pm}{a_\pm+a_\mp}\right)^{p_\mp/2}
            \frac{K_{p_\mp}\left(\sqrt{(a_\mp+a_\pm )b_\mp}\right)}{2K_{p_\mp}\left(\sqrt{a_\mp b_\mp}\right)K_{p_\pm}(\sqrt{a_\pm b_\pm})}
            .
    \end{equation}
\end{proposition}

\begin{proof}
        By definition of the BGIG distribution as a convolution of GIG distributions  \eqref{def:bgig-def}, we have 
        \begin{equation}
        \label{eq:conv-asymptotics}
            \forall x\geq 0,\quad f(x) = \int_0^\infty f_{\GIG_+}(y+x)f_{\GIG_-}(-y)\D y.
        \end{equation}
        It is immediate to see that:
        \begin{equation}
            \quad f_{\GIG_+}(y)
            \underset{y\to+\infty}{\sim} \alpha_+
            y^{p_+-1}
            e^{-\frac{a_+ y}{2}}
        \end{equation}
        where $\alpha_+ :=
            \left(a_+/b_+
            \right)^{p_+/2}(2K_{p_+}(\sqrt{a_+b_+}))^{-1}$.
        Inserting this asymptotic in \eqref{eq:conv-asymptotics} gives:
        \begin{equation}
             f(x) \underset{x\to+\infty}{\sim}\alpha_+ x^{p_+-1}
            e^{-\frac{a_+ x}{2}}\int_0^\infty (1+y/x)^{p_+-1}
            e^{-\frac{a_+ y}{2}} f_{\GIG_-}(-y)\D y.
        \end{equation}
        By the dominated convergence theorem and the definition of the $\mathrm{GIG}_-(a,b,p)$ distribution, we can write:
        \begin{equation}
                \int_0^\infty (1+y/x)^{p_+-1}
                e^{-\frac{a_+ y}{2}}f_{\GIG_-}(-y)\D y
                \underset{x\to+\infty}{\longrightarrow}  \left(\frac{a_+}{a_++a_-}\right)^{p_-/2}
                \frac{K_{p_-}\left(\sqrt{(a_-+a_+)b_-}\right)}{K_{p_-}\left(\sqrt{a_-b_-}\right)}.
        \end{equation}
        Using the symmetry relation in Proposition \ref{prop:symetry} leads to the analog result for $x\to-\infty$.
    \end{proof}

The next proposition is an application of Proposition \ref{prop:asymptotic} to the maximum and minimum of BGIG processes; it will prove to be particulary useful in Section \ref{sec:cal-num-examples} for the calibration of BGIG parameters from empirical data.
    
    \begin{proposition}
    \label{prop:max-min-asymptotic}
        For $n\in\N$, let $(X_i)_{1\leq i\leq n}$ be a sequence of i.i.d. BGIG variables and denote $\displaystyle M_n := \max_{1\leq i\leq n} X_i$ and $\displaystyle m_n := \min_{1\leq i\leq n} X_i$.  We have the following $\mathbb{P}$-almost sure asymptotic behaviours:
        \begin{equation}
            \begin{cases}
                M_n \underset{n\to+\infty}{\sim} \frac{2}{a_+}\ln n,\\
                m_n \underset{n\to+\infty}{\sim} -\frac{2}{a_-}\ln n.
            \end{cases}
        \end{equation}
        
    \end{proposition}
    \begin{proof}
        From Proposition \ref{prop:asymptotic} and the asymptotic behavior of the upper incomplete gamma function (see \cite[Eq. 8.11.2]{NIST:DLMF}), we have 
        \begin{equation}
        \label{eq:asympt-gamma_upper}
        \mathbb{P}\left(X_i>x\right) 
        \underset{x\to+\infty}{\sim} \Gamma(p_+,x)      
        \underset{x\to+\infty}{\sim} Z_+x^{p_+-1}\exp\left(-\frac{a_+x}{2}\right)
        .
        \end{equation}
         Let $\widetilde{M}_n:= \frac{2}{a_+}\ln n$. For $\alpha>0$, we have:
        \begin{equation}
            \mathbb{P}\left(M_n>\alpha \widetilde{M}_n\right) = 1 - \mathbb{P}\left( X_i<\alpha \widetilde{M}_n \quad \forall i = 1 \dots n \right) = 1 - (1-\mathbb{P}(X_i>\alpha M_n))^n.
        \end{equation}
        By \eqref{eq:asympt-gamma_upper}, we have (using $\ln(1+x) = x+\mathcal{O}(x^2) $ for small $x$):
        \begin{equation}
            \begin{aligned}
                \mathbb{P}\left(M_n>\alpha \widetilde{M}_n\right) 
                &= 1 - \exp\left(-n^{1-\alpha}\ln(n)^{p_+-1} \frac{(Z_++o(1))a_+(\alpha \widetilde{M}_n) ^{p_+-1}}{2 } \right)\\
                &\underset{n\to+\infty}{\longrightarrow} 
                \begin{cases}
                    1\quad\mathrm{if}\quad \alpha<1,\\
                     0\quad\mathrm{if}\quad \alpha>1.\\
                \end{cases}
            \end{aligned}
        \end{equation}
        The symmetry relation in Proposition \ref{prop:symetry} leads to the analog result for the minimum.
    \end{proof}

\section{Bilateral generalized inverse Gaussian processes}
\label{sec:BGIG-process}

In this section, we define the BGIG process and provide some of its properties. 
We also study its Blumenthal-Getoor index and its stability under a specific change of measure, as well as the jump structure of its sample paths. 

\subsection{Definition and first properties}
As the BGIG distribution is infinitely divisible (Proposition \ref{prop:inf-div}), there exists an associated Lévy process, said to be generated by the distribution $\BGIG(a_+,b_+,p_+,a_-,b_-,p_-)$. We denote such process by $X=(X_t)_{t\geq 0}$, and we will say, with a slight abuse of language, that $X$ is a $\BGIG(a_+,b_+,p_+,a_-,b_-,p_-)$ process, or simply a BGIG process; as before $(a_+,b_+,a_-,b_- )\in (0,+\infty)^4$ and $(p_+,p_-)\in\R^2$. 

\green{
We define the $\sigma$-algebra 
$\mathcal{F} := \sigma(X_s,s\geq 0)$ and we introduce the augmented filtration $(\mathcal{F}_t)_{t\geq0} :=\sigma((\mathcal{F}_t^X)_{t\geq0}\cup \mathcal{N})$ where $(\mathcal{F}_t^X)_{t\geq0} := (\sigma(X_s, s\in[0,t]))_{t\geq 0}$ is the natural filtration induced by $X$ and $\mathcal N$ are the $\mathbb P$-null sets of $\mathcal F$.  Since $X$ is a L\'evy process, it follows from \cite[Theorem 31]{protter2005stochastic} that $(\mathcal{F}_t)_{t\geq0}$ satisfies the usual hypotheses, \textit{i.e.} is right-continuous and complete.
This allows us to define a filtered probability space $(\Omega, \mathcal{F}, (\mathcal{F}_t)_{t\geq0},\mathbb{P})$,
where $\Omega=\mathbb{D}(\mathbb{R}_+)$ is the space of real-valued càdlàg functions on $\mathbb{R}_+$.
}
\begin{remark}
    From Remark \ref{rem:BGIB_diff}, we see that the difference of two GIG processes is a BGIG process (where the GIG process is the L\'evy process generated by the GIG distribution $\GIG_+(a,b,p)$).
\end{remark}

As for all L\'evy processes, the characteristic function of the BGIG process be simply expressed as:
\begin{equation}
\label{eq:chf-power-t}
    \forall (u,t)\in{\color{black} \mathscr{S}_{\BGIG}}\times (0,+\infty),\quad \Phi(u,t) = \Phi(u)^t 
\end{equation}
where $\Phi$ is the characteristic function of the generating distribution \ref{eq:bgig-chf}. Of course, it follows that $X_1\sim \BGIG(a_+,b_+,p_+,a_-,b_-,p_-)$. The Lévy-Khintchine representation for the characteristic function \eqref{eq:chf-power-t} follows immediately from \eqref{eq:LK} and reads
\begin{equation}
        \forall (u,t)\in{\color{black} \mathscr{S}_{\BGIG}}\times(0,+\infty), \quad \Phi(u,t) = \exp\left(t \int_\R \left( e^{\mi u x}-1   \right) \pi(\D x)\right).
\end{equation}
where $\pi$ is the L\'evy measure defined in \eqref{eq:levy-meaure-BGIG}. 
We display an example of its transition densities in Figure \ref{fig:density-time} for different values of $t$, obtained by inverting the characteristic function in eq. \ref{eq:chf-power-t}.


\begin{figure}[H]
        \centering
        \includegraphics[scale = 0.9]{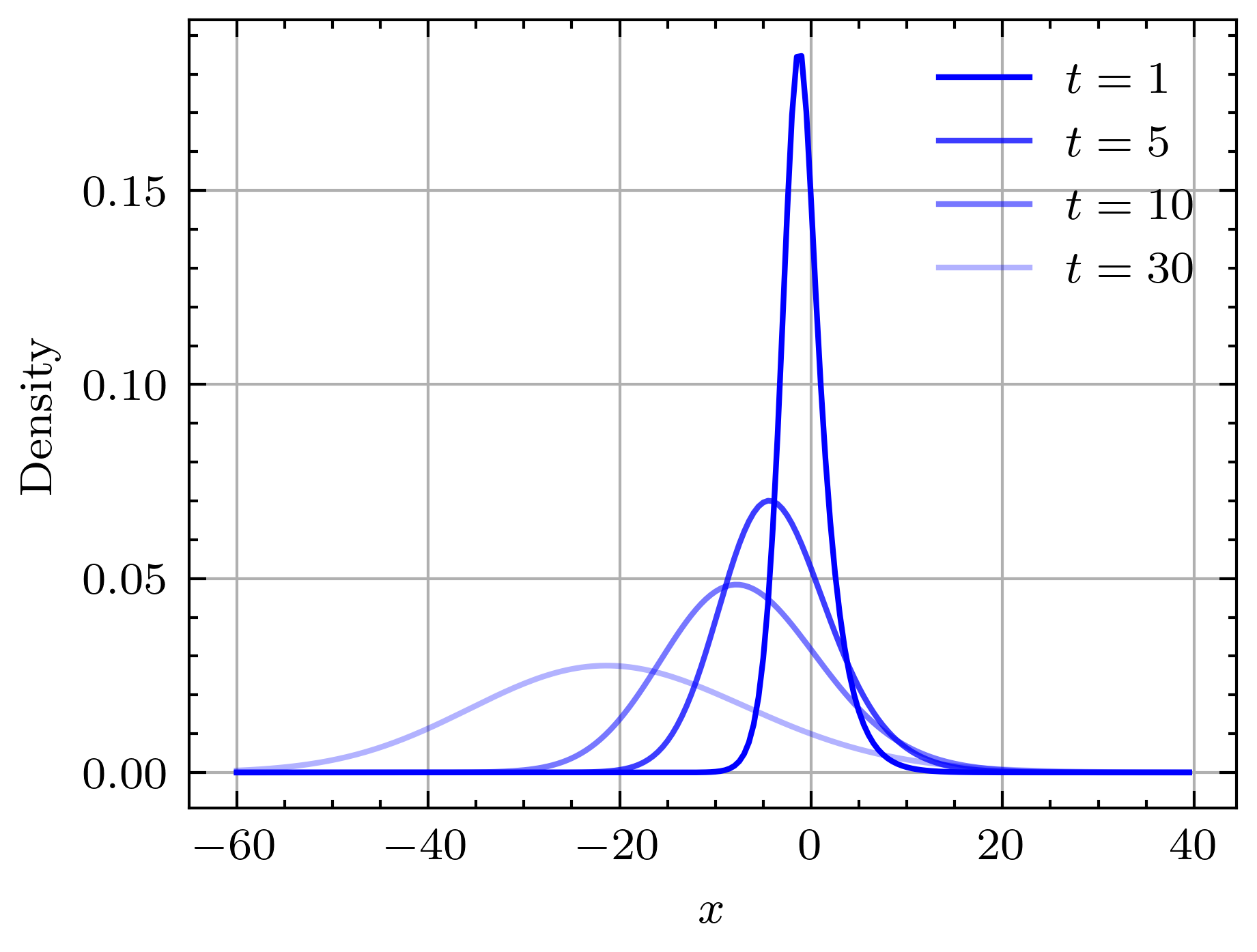}
        \caption{Transition densities for a $\BGIG(1,2,1,3,4,5)$-process and different values of $t$.}
        \label{fig:density-time}
\end{figure}

\begin{proposition}
    \label{prop:pure-jumps}
    If $X$ is a BGIG process, then:
    \begin{enumerate}
        \item $X$ is a finite variation process,
        \item $X$ has infinite activity,
        \item $X$ is equal to the sum of its jumps, i.e., 
        \begin{equation}
            \forall t>0,\quad X_t = \sum_{s\leq t} \Delta X_s.
        \end{equation}
    \end{enumerate}
\end{proposition}

\begin{proof}
The L\'evy measure $\pi$ satisfies $\pi(\R)= \infty$ and
$
    \int_{|x|<1} |x|\pi(\D x) <\infty
$
because of the asymptotic for the Jaeger integral obtained in Lemma \ref{lemma:jaeger-bounds}. The process $X$ is therefore of type B in the sense of \cite{Satolévy}. Then, points (1), (2) and (3) follow from \cite[Th. 21.9, 21.3, 19.3]{Satolévy} respectively.
\end{proof}

\subsection{Blumenthal-Getoor index}

The Blumenthal-Getoor index $\beta(X)$ (see \cite{Feller-BG}) of a L\'evy process $X$ characterizes its small jumps. Formally, it is defined as:
\begin{equation}
    \beta(X) := \inf\left\{p>0,\quad \int_{|x|<1}|x|^p \pi_X(\D x) < \infty\right\}
    ,
\end{equation}
where $\pi_X$ denotes the L\'evy measure of $X$. 
It is known that a Lévy process is of finite variation if and only if $\beta(X)<1$. For example, a Tempered stable process has Blumenthal-Getoor index equal to $\max(\beta_+,\beta_-)$ (see \cite{KuchlerTS}). As consequence, a bilateral Gamma process has a Blumenthal-Getoor index of 0.  

\begin{proposition}
\label{prop:BG-index}
    If $X$ a BGIG process, we have $\beta(X) = 1/2$.
\end{proposition}

\begin{proof}
    Let $p>0$; by definition of the BGIG L\'evy measure \eqref{eq:levy-meaure-BGIG}, we can write:
    \begin{multline}\label{xp}
        \int_{-1}^1|x|^p\pi(\D x) =
         \int_{0}^1\frac{e^{-a_+x/2}}{x^{1-p}}\left(\mathscr{I}(x/2b_+,p_+) + \max(p_+,0)\right)\D x
         \\
        + 
        \int_{0}^1\frac{e^{-a_-x/2}}{x^{1-p}}\left(\mathscr{I}(x/2b_-,p_-)+ \max(p_-,0)\right)\D x.
    \end{multline}
    By the asymptotic given in Lemma \ref{lemma:jaeger-bounds}, the two integrals in the right hand side of \eqref{xp} are finite if and only if $p>\max(1/2,0) = 1/2$ and the desired results follows. 
\end{proof}

\begin{proposition}
    If X is a BGIG process, then we have the following almost-sure convergences:
    \begin{equation}
        \left\{
        \begin{aligned}
        &\forall p>\frac{1}{2},\quad \lim_{t\to 0} t^{-1/p}X_t = 0, \\
        &\forall p<\frac{1}{2},\quad \limsup_{t\to 0} t^{-1/p}X_t = \infty.
        \end{aligned}
        \right.
    \end{equation}
\end{proposition}
\begin{proof}
    This is a direct consequence of Proposition \ref{prop:BG-index} and \cite[Thm. 3.1, Thm. 3.3]{Feller-BG}.
\end{proof}


\subsection{Change of measure}
We provide here some properties of BGIG distributions under change of measures, which, as we will discuss in Section \ref{sec:exp-bgig}, are important for financial applications. Let us first define $\nu$, the Radon-\textcolor{black}{Nikodym} derivative of the two Lévy measure $\pi$ and $\widetilde{\pi}$ of two BGIG distributions. We have:
\begin{multline}\label{eq:nu}
    \forall x\in\R,\quad \nu(\D x) := e^{-(\widetilde{a}_+-a_+)x/2}\left(
    \frac{\mathscr{I}(x/2\widetilde{b}_+;\widetilde{p}_+) + \max(\widetilde{p}_+,0)}
    {\mathscr{I}(x/2b_+;p_+) + \max(p_+,0)}\right)
    \mathbbm{1}_{(0,\infty)}(x)\D x\\
    +  e^{-(\widetilde{a}_--a_-)|x|/2}\left(
    \frac{\mathscr{I}(|x|/2\widetilde{b}_-;\widetilde{p}_-) + \max(\widetilde{p}_-,0)}
    {\mathscr{I}(|x|/2b_-;p_-) + \max(p_-,0)}\right)
    \mathbbm{1}_{(-\infty,0)}(x)\D x.\\
\end{multline}

\begin{proposition}
    \label{prop:change_of_measure}
    Let $X$ be a $\BGIG(a_+,b_+,p_+,a_-,b_-,p_-)$-process and let  $(\widetilde{a}_+,\widetilde{b}_+,\\\widetilde{p}_+,\widetilde{a}_-,\widetilde{b}_-,\widetilde{p}_-)$ be six real numbers in the acceptable parameter space of BGIG distributions. The following statements are equivalent:

    \begin{enumerate}
        \item There exists a measure $\mathbb{Q}\overset{\mathrm{loc}}{\sim}\mathbb{P}$ under which $X$ is a BGIG process with parameters $(\widetilde{a}_+,\widetilde{b}_+,\widetilde{p}_+,\widetilde{a}_-,\widetilde{b}_-,\widetilde{p}_-)$.
        \item $b_\pm = \widetilde{b}_\pm$ and $p_\pm = \widetilde{p}_\pm$.
    \end{enumerate}
\end{proposition}

\begin{proof}
    We use Theorem 33.1 of \cite{Satolévy}.
    The only condition (the others are immediately satisfied) to fullfill is:
    \begin{equation}\label{eq:integrability}
        \int_\R\left(1-\sqrt{\nu(x)})\right)^2\pi (\D x) <\infty.
    \end{equation}
    Knowing $\pi$ and $\nu$ (see \eqref{eq:levy-meaure-BGIG} and \eqref{eq:nu}), this integral can be rewritten as:
    \begin{multline}
        \int_{0}^\infty x^{-1}\left(h(x;a_+, b_+,p_+)
        - h(x;\widetilde{a}_+,\widetilde{b}_+, \widetilde{p}_+\right)^2\D x
        \\
        +\int_{0}^\infty x^{-1}\left(h(x;a_-, b_-,p_-)
        - h(x;\widetilde{a}_-,\widetilde{b}_-, \widetilde{p}_-\right)^2\D x\\
    \end{multline}
    where $h(x;a,b, p) := e^{-ax/4}\sqrt{\mathscr{I}(x/2b;p) + \max(p,0)}$.
    The integrability condition \eqref{eq:integrability} is satisfied if and only if:
    \begin{equation}
        \mathscr{I}(x/2b_\pm , p_\pm) + \max(p_\pm,0)
         =\mathscr{I}(x/2\widetilde{b}_\pm , \widetilde{p}_\pm) + \max(\widetilde{p}_\pm,0)
         ,
    \end{equation}
    for $x$ in a neighborhood of 0. The asymptotic for $x\to 0$ given in \textcolor{black}{Lemma} \ref{lemma:jaeger-bounds} gives that this is true if and only if $ b_\pm =\widetilde{b}_\pm$ and $p_\pm =\widetilde{p}_\pm$.
\end{proof}
\begin{remark}
As we will see in Section \ref{sec:exp-bgig}, this proposition is verified for the Esscher transform; this will allow us to construct a risk-neutral version of the (exponential) BGIG process, opening the way to the valuation of financial instruments. 
\end{remark}

    \subsection{Jump structure of the sample paths}
In this subsection we prove a result that can be useful for estimating the parameters $b_\pm$ of a BGIG process from its sample path; this result is particularly interesting if one wishes to infer BGIG parameters from historical data featuring a very short time step (for instance, tick by tick data). First, let us consider a process $X\sim \GIG(a,b,p)$ and let us define, for $n\geq 1$ and $T>0$, 
\begin{equation}
    U_n := \frac{1}{nT}\#\left\{t\leq T,\quad \Delta X_t\geq n^{-2}\right\}.
\end{equation}

    \begin{proposition}
        \label{prop:estim-b}
        The sequence of random variables $(U_n)_{n\geq 1}$ converges to $\sqrt{\frac{b}{2\pi}}$ almost surely. 
    \end{proposition}

    \begin{proof}
        The random measure $\mu_X$ of the jumps of the BGIG process is a Poisson random measure, whose intensity measure $\nu$ is given by (see \cite{Satolévy}):
            \begin{equation}\label{eq:nu_intensity}
                \nu(\D t,\D x) = \frac{e^{-ax/2}}{x}\mathscr{I}(x/2b)\mathbbm{1}_{(0,+\infty)}(x)\D x \D t + \max(0,p)\frac{e^{-ax/2}}{x}\mathbbm{1}_{(0,+\infty)}(x)\D x \D t.
            \end{equation}
        We define $\widetilde{U}_n$ as:
        \begin{equation}
            \widetilde{U}_n = \frac{1}{T}\mu_X([0,T]\times [(n+1)^{-2},n^{-2}]).
        \end{equation}
        By definition, the \textcolor{black}{$\widetilde{U}_n $} are independent. Let us prove that the first two moments of the \textcolor{black}{$\widetilde{U}_n $}  are finite. By a change of variable $x=u^{-1/2}-n$, we can write:
        \begin{multline}
                \mathbb{E}(\widetilde{U}_n) 
                = 2\int_0^1 \frac{e^{-a(x+n)^{-2}/2}}{(x+n)^{-2}}\mathscr{I}((x+n)^{-2}/2b) (x+n)^{-3} \D x\\
                \quad + 2\max(0,p)\int_{0}^{1}\frac{e^{-a(x+n)^{-2}/2}}{(x+n)^{-2}}
                (x+n)^{-3}\D x.
        \end{multline}
        The second integral goes to $0$ when $n\to\infty$ by the dominated convergence theorem. Furthermore, using the asymptotic of Lemma \ref{lemma:jaeger-bounds} 
        and the dominated convergence theorem again, we have:
        \begin{equation}
            \int_0^1 \frac{e^{-a(x+n)^{-2}/2}}{(x+n)^{-2}}\mathscr{I}((x+n)^{-2}/2b) (x+n)^{-3} \D x \underset{n\to\infty}{\longrightarrow}\sqrt{\frac{b}{2\pi}}.
        \end{equation}
        Thus, we have $\mathbb{E}(\widetilde{U}_n)\to \sqrt{b/2\pi}$ and similarly $\mathrm{Var}(\widetilde{U}_n)\to (1/T)\sqrt{b/2\pi}$ when $n\to\infty$. Consequently, $ \sum_{n\geq1} \mathrm{Var}(\widetilde{U}_n)/n^2<\infty$.
        Applying Kolmogorov's strong law of large numbers and noticing that $\mu_X([0,T]\times [1,\infty))<\infty$, we conclude:
        \begin{equation}
            \lim_{n\to\infty} U_n= \lim_{n\to\infty}\frac{1}{n}\sum_{i=1}^n \widetilde{U}_n+\lim_{n\to\infty} \frac{1}{nT}\mu_X([0,T]\times [1,\infty)) = \sqrt{\frac{b}{2\pi}}
            .
        \end{equation}
    \end{proof}
    Let $X$ be a $\BGIG(a_+,b_+,p_+,a_-,b_-,p_-)$ process and define the random variables:
    \begin{equation}
        U_n^+ := \frac{1}{nT}\#\{t\leq T,\quad \Delta X_t\geq n^{-2}\} 
        ,
        \quad
        U_n^- := \frac{1}{nT}\#\{t\leq T,\quad \Delta X_t\leq -n^{-2}\}.     
    \end{equation}

    \begin{proposition}
        The following $\mathbb{P}$-almost sure convergence holds:
        \begin{equation}
            \lim_{n\to\infty}(U_n^+,U_n^-) = \left(\sqrt{\frac{b_+}{2\pi}},\sqrt{\frac{b_-}{2\pi}}\right).
        \end{equation}
    \end{proposition}

    \begin{proof}
        The proof is straightforward with Proposition \ref{prop:estim-b} and since a BGIG process can be decomposed as $X=X^+-X^-$ with $X^\pm$ two independent GIG processes.
    \end{proof}

\subsection{Simulation of BGIG processes}
\label{subsec:simul}

For a stochastic process to gain popularity in quantitative finance (notably from the point of view of market practitioners), it is important that one can efficiently simulate it; the reason is that, in order to value very exotic instruments (that can not be reached via (semi-)closed formulas), one typically relies on Monte-Carlo simulations, hence the necessity of generating a possibly large number of sample paths. Let us therefore provide two different methods for simulating BGIG processes.


\subsubsection{\green{Simulation for integer times}}
First, a simulation of the BGIG process at a collection of integer times $\{t_i\}_{i = 1 \dots N}\\\in \N^N$ can be achieved very simply. Indeed, since BGIG process are L\'evy process and therefore have independent increments, in this case we can write:
    \begin{equation}
        \forall i = 1 \dots N,\quad X_{t_i} = \sum_{k=0}^{t_i} \left(X_k-X_{k-1}\right)
        \overset{\mathcal{L}}{=}\sum_{k=1}^{t_i} X_1^{(k)}
    \end{equation}
    \green{where $X_1^{(k)}:=X_{1,+}^{(k)}-X_{1,-}^{(k)}$ and $X_{1,\pm}^{(k)}$
    are i.i.d. random variables following the generating $\GIG_+(a_\pm,b_\pm,p_\pm)$ distributions (see Remark \ref{rem:BGIB_diff})}. We illustrate this method in Figure \ref{fig:simulations}; the theoretical density are obtained by inverting the characteristic function in \eqref{eq:chf-power-t}.

    \begin{figure}
        \begin{subfigure}{.5\textwidth}
          \centering
          \includegraphics[width=.9\linewidth]{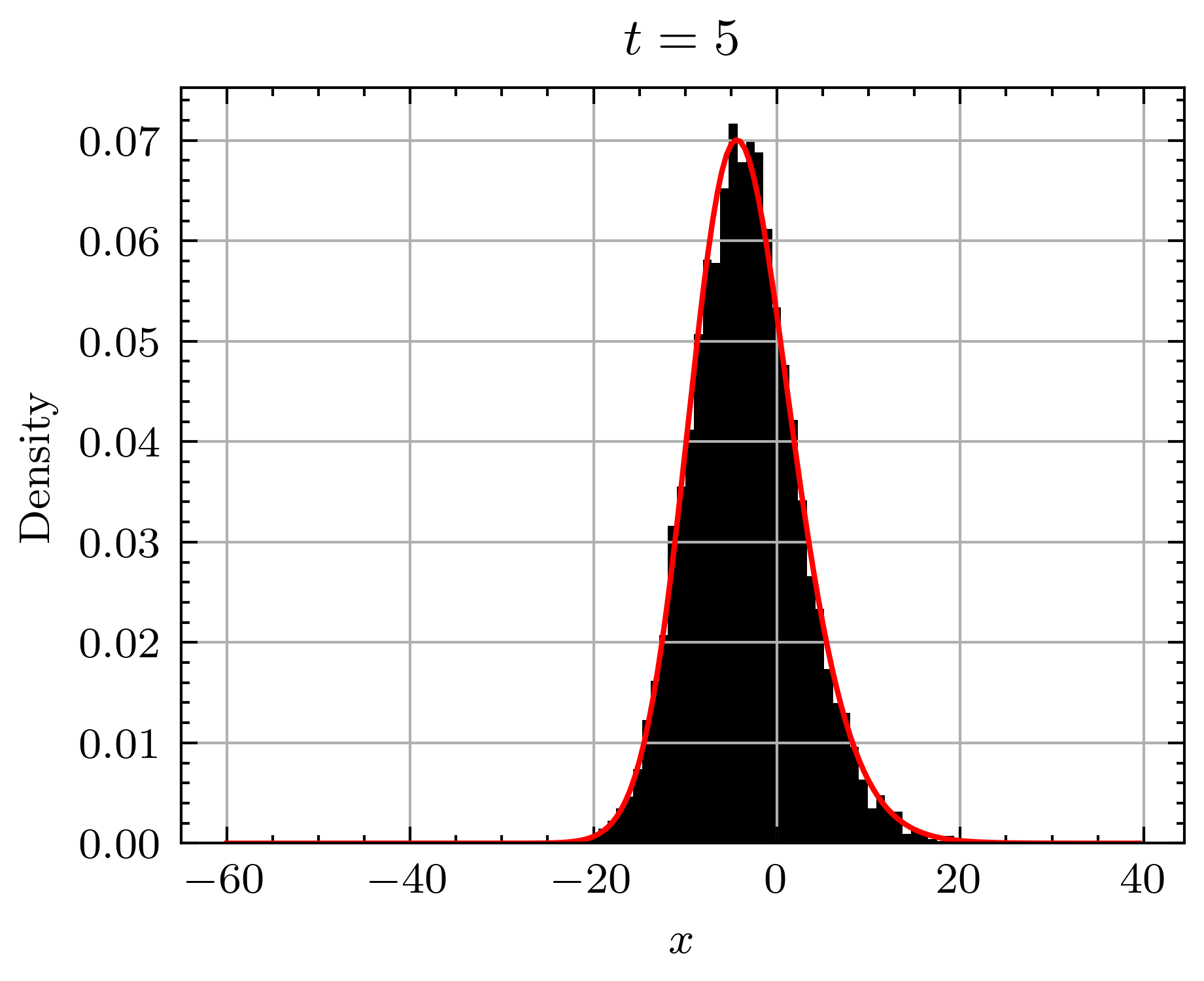}
          \caption{$t=5$}
        \end{subfigure}%
        \begin{subfigure}{.5\textwidth}
          \centering
          \includegraphics[width=.9\linewidth]{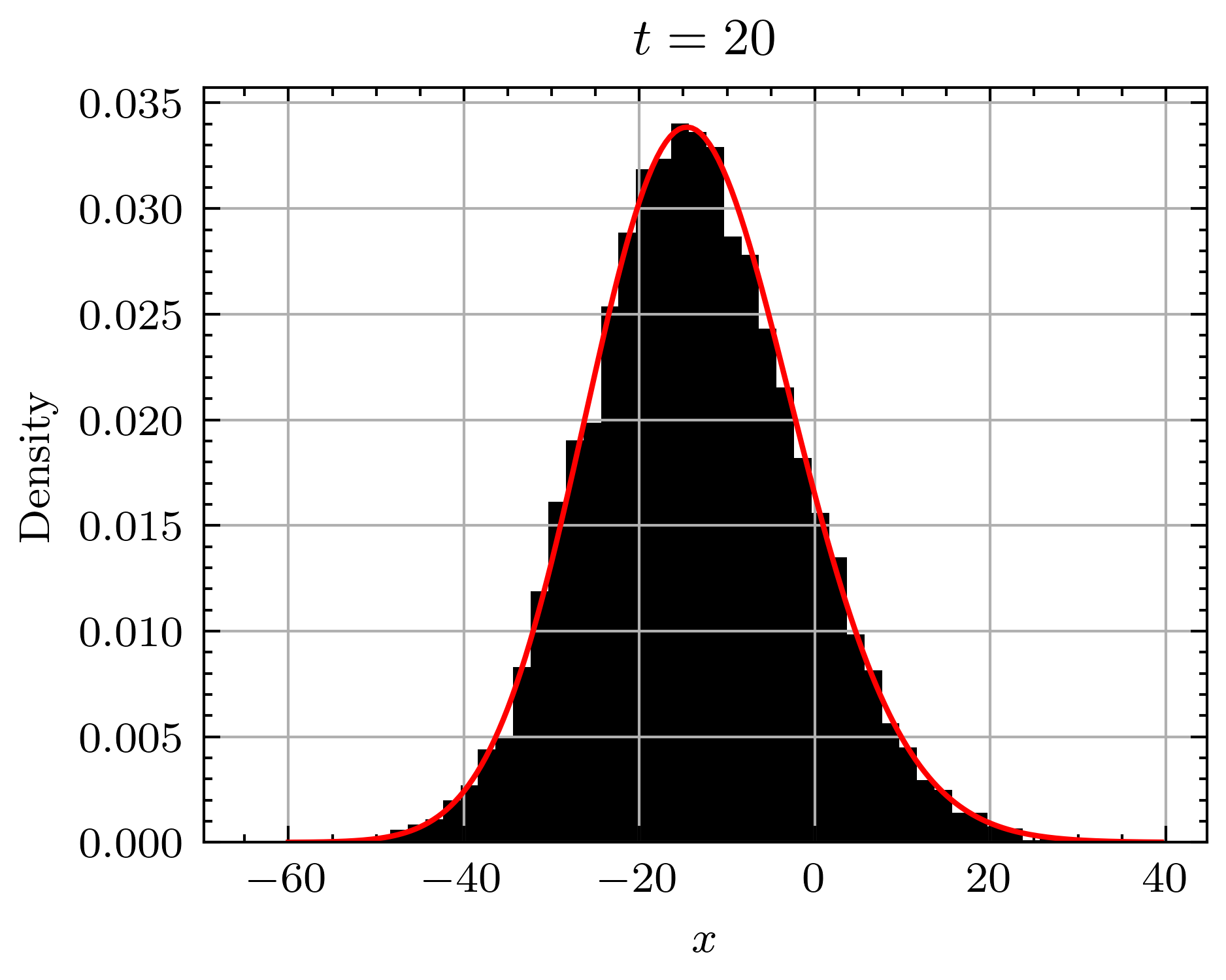}
          \caption{$t=20$}
        \end{subfigure}
        \caption{Histogram of 10 000 simulations of a BGIG(1,2,1,3,4,5)-process for $t=5$ and $t=20$, and the corresponding theoretical densities.}
        \label{fig:simulations}
    \end{figure}

\subsubsection{\green{Simulation for arbitrary times}}
    Since BGIG distributions are not convolution-closed, this method cannot be applied in general for $\{t_i\}_{i = 1 \dots N}\in \R_+^N$. In this case, recalling that a BGIG process $X$ can be written $X:=X_+-X_-$ where $X_\pm$ are two independent GIG processes, we can use the simulation algorithm developed in \cite{Godsill} for each GIG process to obtain a simulation for the BGIG process $X$. {\color{black}In Figure \ref{fig:simulations_process}, we display a set of 50 simulations of the BGIG process using this methodology, as well as the cumulated returns of the S\&P 500 Index over a one-year period; we can observe the similarities between the simulations (in black) and the realized returns (in black), demonstrating the ability of the BGIG process to reproduce real-world market data. The BGIG parameters are calibrated from the daily log-returns of the index (see Section \ref{sec:cal-num-examples} for details on the calibration procedure).}


    \begin{figure}
        \centering
        \includegraphics[width=0.85\linewidth]{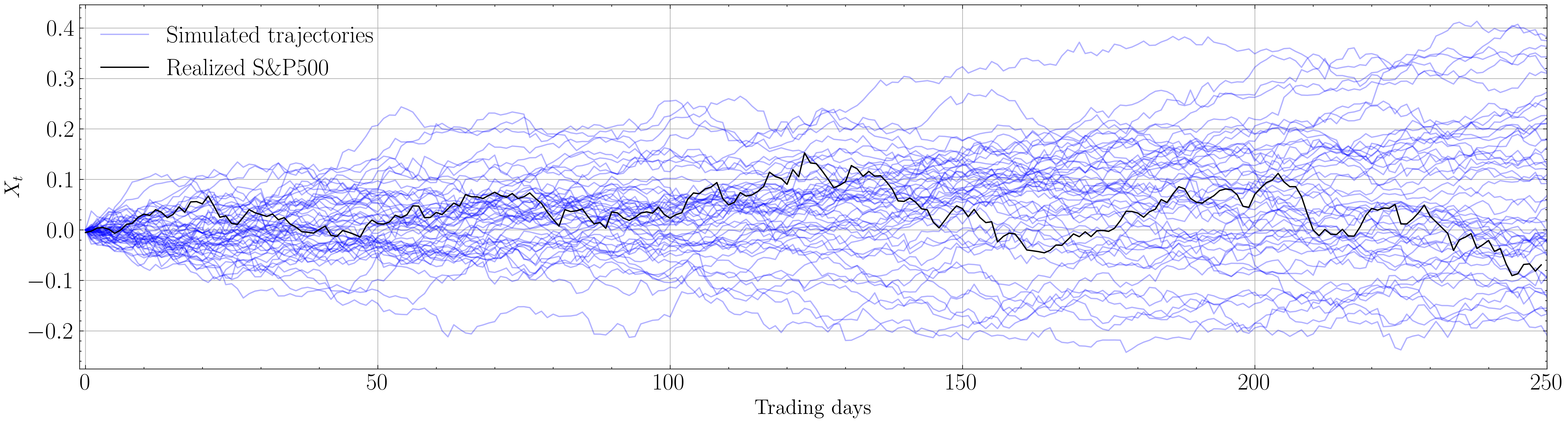}
        \caption{{\color{black}{50 simulations of the BGIG process, with market-calibrated parameters (eq. \ref{eq:cal-params}).}}}
        \label{fig:simulations_process}
    \end{figure}

\section{The exp-BGIG stock market model}
    \label{sec:exp-bgig}

Let us now demonstrate how BGIG processes can be used in mathematical finance, in particular for the purpose of asset returns modeling and derivatives pricing.

\subsection{Definition and change of measure}

    In continuous time finance, a popular approach to modeling asset prices is via the introduction of exponential L\'evy processes, because they allow to capture realistic features (most notably the occurrence of jumps) while remaining computationally tractable; see, among many other references, \cite{barndorff2001lévy,Schoutens03,Cont04}. Here, we will assume that the exponential L\'evy process is an exponential BGIG (exp-BGIG) process;  More precisely, {\color{black}we fix $T\geq 0$ a finite time horizon} and introduce the stochastic process $S=\{ S_t \}_{t\in[0,T]}$ describing the price of a financial asset (a stock price, an index, an interest rate...), that is defined as:
\begin{equation}\label{eq:expBGIB}
    \forall {\color{black}t\in[0,T]}, \quad S_t = S_0 e^{X_t} , \quad S_0 >0,
\end{equation}
where $X=\{X_t\}_{\color{black}{t\in[0,T]}}$ is a BGIG process. One essential point when one wishes to conduct pricing of financial instruments written on $S$ is to find a measure, called risk-neutral measure \green{or equivalent martingale measure (EMM)} under which the discounted asset price process $\widetilde{S} := \{e^{-rt}S_t\}_{\color{black}{t\in[0,T]}}$ is a martingale \cite{EberleinKalsen,Schoutens03}; here $r$ is the risk-free rate that is assumed to be deterministic and continuously compounded. 
Let us start by mentioning the following equivalence: 

\begin{proposition}
\label{prop:martingale}
    Assume $a_+>2$ and
    let $S$ be an exp-BGIG process. Then $S$ is a martingale if and only if the following equality holds:
    \begin{multline}\label{S_martingale}
        \left(\frac{a_+}{a_+-2}\right)^{p_+/2}
        \frac{K_{p_+}\left(\sqrt{b_+(a_+-2 )}\right)}{K_{p_+}\(\sqrt{b_+a_+}\)}
        \times
        \left(\frac{a_-}{a_- +2}\right)^{p_-/2}\frac{K_{p_-}\(\sqrt{b_-(a_- +2 )}\)}{K_{p_-}\(\sqrt{b_-a_-}\)} = 1.
    \end{multline}
\end{proposition}

\begin{proof}
    For exponential Lévy models, we have by \cite[Prop. 3.17]{Cont04} that the martingale condition is equivalent to have $\mathbb{E}[e^{X_1}]=1$. \textcolor{black}{Since $\Imag[-\mi] = -1 >-a_+/2$, we have $-\mi\in\mathscr{S}_{\BGIG}$ and therefore we can write $\mathbb{E}[e^{X_1}]=\Phi(-\mi,1)$ and the result follows.}
    
\end{proof}
\begin{remark}\label{rem:S_martingale}
    Of course for the discounted price process $\tilde{S}$, the right hand side of \eqref{S_martingale} is replaced by $e^{r}$.
\end{remark}

\green{Let us discuss more in details the existence of an EMM for $\widetilde{S}$. We first note that, in the case of markets driven by exponential Lévy process, the notion of martingale is equivalent to the notion of $\sigma$-martingale (see \cite[p. 22]{cherny2002change}). Moreover, such markets are in general incomplete at the exceptions of two particular cases: when the underlying Lévy process is a Brownian motion or a Poisson process with single jump size
(see \cite[p. 343]{Cont04}). Outside these two cases, either an EMM does not exist, or it is not unique; we know that the existence of an EMM is actually equivalent to the requirement that the L\'evy process is not a.s. monotone (see \cite[Prop. 9.9]{Cont04}), which is indeed the case here as BGIG processes feature both positive and negative jumps. 
}
One classical way \green{to construct an EMM} is via the so-called Esscher transform (see \cite{EsscherPricing}), that provides a straightforward procedure to build a risk-neutral measure. Indeed, let us define $\theta^*$ to be the solution of the equation:
\begin{equation}
    \label{eq:theta-star}
    \log \Phi_X(-\mi(\theta^*+1)) - \log \Phi_X(-\mi\theta^*) = r.
\end{equation}
where $\Phi_X$ is the characteristic function of $X$, and introduce the {\it Esscher transform}  $\mathbb P^*$ of the $\mathbb P$  measure defined by the following Radon-Nikodym derivative
\begin{equation}\label{def:Esscher}
    \frac{d\mathbb{P}^*_{\mathcal{F}_t}}{d \mathbb P_{\mathcal{F}_t}}  
     =  
    \frac{e^{\theta^* X_t}}{\mathbb E \left[ e^{\theta^* X_t} \right]}
    \, = \,
    e^{\theta^* X_t - \ln \Phi_X(-\mi \theta^*) t}
    .
 \end{equation}
Then, following Remark \ref{rem:S_martingale}, $\tilde{S}$ is a martingale under $\mathbb P^*$. We note that, as a direct consequence of Definition \ref{def:Esscher}, the characteristic function of the process $X$ under $\mathbb P^*$ can be explicitly written as
\begin{equation}\label{char_Esscher}
    \Phi_X^*(u,t) 
    = \mathbb E^{\mathbb P ^*} \left[ e^{\mi u X_t} \right]
    =
    \frac{\Phi_X(u- \mi \theta^*,t)}{\Phi_X(-\mi\theta^*,t)}
     . 
\end{equation}

\begin{proposition}
\label{prop:esccher}
    Let $S$ be an exp-BGIG process process with parameters $(a_+,b_+,p_+,a_-,b_-,p_-)$. Under the Esscher measure $\mathbb P^*$, $S$ is still an exp-BGIG process with parameters $(a_+-2\theta^*,b_+,p_+,\\a_-+2\theta^*,b_-,p_-)$.
\end{proposition}

\begin{proof}
    By the properties of the convolution, it is sufficient to prove the result for a one sided GIG distribution. By the time-scaling property of the L\'evy characteristic functions, it is sufficient to prove it for $t=1$. For $a,b>0$ and $p\in\R$, we have:
    \begin{multline}
        \forall u\in\R,\quad \Phi^*_{\GIG_+}(u) = \frac{\Phi_{\GIG_+}(u- \mi \theta^*,t)}{\Phi_{\GIG_+}(-\mi\theta^*,t)}
        = 
        \left(
        \frac{a-2\theta^*}{a-2\theta^*-\mi u }\right)^{p/2}
        \frac{K_{p}\left(\sqrt{b(a-2\theta^* -\mi 2 u)}\right)}{K_{p}\(\sqrt{b(a-2\theta^*)}\)}.
    \end{multline}
    Thus, under the Esscher measure, a $\GIG_+(a,b,p)$ distribution is a $\GIG_+(a-2\theta^*,b,p)$ distribution. The desired result immediately follows.
\end{proof}
This proposition obviously satisfies the parameters condition of Proposition \ref{prop:change_of_measure} and the process thus obtained satisfies martingale condition of Proposition \ref{prop:martingale} if $a_+-2\theta^*>2$. We will see in Section \ref{subsec:SP} that this condition is satisfied in practice. 

\subsection{Pricing Methods}

Once a risk neutral measure $\mathbb P^*$ has been constructed, it is possible to apply the risk-neutral pricing theory (see \cite{EberleinKalsen} or any classical monograph on quantitative finance) to compute the value of derivatives, that is, of financial instruments written on the underlying asset price $S$. Such instruments (typically, options) are characterized by the payoff $g(S)$ they deliver, that can be a function of either the final underlying price $S_T$ or of the whole realization of the process on $[0,T]$; their price at time $t\in[0,T]$ corresponds to the risk-neutral expectation of the discounted payoff, that is:
    \begin{equation}\label{eq:expect}
        \color{black}
        \forall t\in [0,T], \quad 
        V_t = \mathbb{E}^{\mathbb{P}^*} [ e^{-r(T-t)} g(S_T)|\mathcal{F}_t].
    \end{equation}
In the following, we will conduct valuations on $t=0$, and we will focus on European call options, which are among the most widely traded derivatives and whose payoff is given by
\begin{equation}
    g_{\mathrm{\textcolor{black}{call}}}(S) = \max (S_T - K, 0) = (S_T-K)^+
\end{equation}
for some predetermined $K>0$, called strike price. We will denote their price by $C(S_0,K,T)$ to highlight the dependence on current (spot) asset value $S_0$, strike $K$ and maturity (or expiry) $T$. \textcolor{black}{Similarly, European put options give the right to sell the asset and thus pays $(K-S_T)^+$ at maturity.
}

{\color{black}
Several methods exist to evaluate the expectation \eqref{eq:expect}; 
in this paper, we will implement either the well-known Monte Carlo (MC) approach (see \cite{Glasserman03} for a complete introduction to MC methods in finance), or the so-called Fourier pricing approach. Fourier pricing is particularly well suited to L\'evy models, because it relies on the characteristic function of the process (instead of its density), and this function is generally known under a simple form. For instance, one can use the Lewis-Lipton (LL) formula \cite{Lewis01}:
    \begin{equation}\label{eq:Lewis}
        C(S_0,K,T) = \frac{-Ke^{-rT}}{2\pi}\int_{\mi v_1 -\infty}^{\mi v_1 +\infty}\frac{e^{-\mi zk}\Phi^{*}(-z,T)}{z^2-\mi z}\D z
        ,
    \end{equation}
    where $k =\ln (S_0/K) + rT$, $v_1>1$ and $\Phi^{*}$ is the BGIG characteristic function under $\mathbb P^*$. Note that many extensions of this idea exist that improve its accuracy and performance, such as the Carr-Madan formula \cite{CGMY} or the COS \cite{fang2009novel}, CONV \cite{CONV} or PROJ \cite{kirkby2015efficient} methods.
}

\section{Calibration procedure and numerical examples}
\label{sec:cal-num-examples}


We terminate the paper by performing a calibration of the exp-BGIG model \eqref{eq:expBGIB} over real market data, and by pricing European call and put options under both Monte Carlo and Fourier methods.

\subsection{Calibration methodology}
\label{subsec:calibration_methodo}




In this subsection, we provide a procedure to calibrate the exp-BGIG model over empirical data, that is to infer the parameters $(a_+,b_+,p_+,a_-,b_-,p_-$) from a given data set composed of daily close prices, for a set of trading dates $(t_1,t_2,\dots, t_n)$ with for simplicity constant spacing $t_{k+1}-t_{k}=\delta>0$. We denote the corresponding underlying prices by $S_k=S_{t_k}$ and we define the series of log-returns as $L = (\ln(S_{k+1}/S_k))_{k=1\dots n}$. 
Let us remark that
$
    \ln\left(S_{k+1}/{S_{k}}\right)
    =
    X_{t_{k+1}}  - X_{t_{k}}
    \overset{\mathcal{L}}{=} 
    X_{\delta}
$
where the second equality (in law) is a consequence of the stationarity of increments of L\'evy processes. In particular when $\delta =1$, $L$ is a collection of $n$ independent $\BGIG(a_+,b_+,p_+,a_-,b_-,p_-)$ random variables.

Our calibration procedure can be divided into two steps: first we estimate $a_+$ and $a_-$ thanks to the asymptotic properties of BGIG distributions (Proposition \ref{prop:max-min-asymptotic}), then we estimate $b_+$, $b_-$,  $p_+$ and $p_-$ via the method of moments. 

\subsubsection{Estimating \texorpdfstring{$a_\pm$}{apm}}
\label{sec:estimation-a-pm}
Let us denote $m_n = \min_{1\leq i \leq n} L_i$ and $M_n := \max_{1\leq i \leq n} L_i$. Taking $\delta=1$ and using Proposition \ref{prop:max-min-asymptotic} we have
$m_n \underset{n\to\infty}{\sim} -\frac{2}{a_-} \ln n$ and $M_n \underset{n\to\infty}{\sim} \frac{2}{a_+}\ln n$.
Our estimators $\widehat{a}_+$ and $\widehat{a}_-$ will therefore be simply defined as:
\begin{equation}
    \widehat{a}_+ := \frac{2 \ln n }{M_n}\quad\mathrm{and}\quad\widehat{a}_- := -\frac{2 \ln n }{m_n}. 
\end{equation}
These estimators are directly linked to the nature of the BGIG distribution and thus have the advantage to be robust. Moreover they allow to restrict the subsequent method of moments to the estimation of four parameters instead of six, which reduces the complexity of the system to be solved.


\subsubsection{Estimating \texorpdfstring{$b_\pm$}{bpm} and \texorpdfstring{$p_\pm$}{ppm}}
\label{sec:estimation-b-p-pm}
The parameters $b_\pm$ and $p_\pm$ are estimated thanks to the relations between moments and cumulants. For $\delta>0$, all the returns $L_k$ are i.i.d and we have:
\begin{equation}
\label{eq:link-cumulants-moments}
    \mathbb{E}(L_k)     = \delta\kappa_1,~  
    \mathrm{Var}(L_k) = \delta\kappa_2,~
    \gamma_1(L_k)     = \dfrac{\kappa_3}{\kappa_2^{3/2} \sqrt{\delta}},~
    \gamma_2(L_k)     = 3 + \dfrac{\kappa_4}{\kappa_2^2 \delta}.
\end{equation}
where $\gamma_1 (L_k)$ (rep. $\gamma_2 (L_k)$) denotes the skewness (resp. kurtosis) of $L_k$, and where the BGIG cumulants $\kappa_j$ are the ones given in Section \ref{subsubsec:cumul}. 
The estimators $\mathfrak{m}_j$ of these four central and standardized moments are defined as:
\begin{equation}
\left\{
\begin{array}{rcl@{\quad}rcl}
    \displaystyle \mathfrak{m}_1 &=& \dfrac{1}{n} \sum_{k=0}^n
    L_k ,
    & \displaystyle \mathfrak{m}_2 &=& \dfrac{1}{n} \sum_{k=0}^n (L_k - \mathfrak{m}_1)^2, \\[1.2ex]
    \displaystyle \mathfrak{m}_3 &=& \dfrac{1}{n \mathfrak{m}_2^{3/2}} \sum_{k=0}^n (L_k - \mathfrak{m}_1)^3, 
    & \displaystyle \mathfrak{m}_4 &=& \dfrac{1}{n \mathfrak{m}_2^{2}} \sum_{k=0}^n (L_k - \mathfrak{m}_1)^4.
\end{array}
\right.
\end{equation}
Cumulants in \eqref{eq:link-cumulants-moments} are functions of the parameters $a_\pm$, $b_\pm$ and $p_\pm$:  we thus look for a vector $(b_+,p_+,b_-,p_-)\in ((0,+\infty)\times \R)^2$ solving:
\begin{equation}
    \left\{
    \begin{aligned}
        \displaystyle \mathfrak{m}_1 &= \delta\kappa_1(\widehat{a}_+,b_+, p_+,\widehat{a}_-,b_-,p_-) ,\\
        \displaystyle \mathfrak{m}_2 &= \delta\kappa_2(\widehat{a}_+,b_+, p_+,\widehat{a}_-,b_-,p_-),\\
        \displaystyle \mathfrak{m}_3 &= \frac{\kappa_3(\widehat{a}_+,b_+, p_+,\widehat{a}_-,b_-,p_-)}{\kappa_2(\widehat{a}_+,b_+, p_+,\widehat{a}_-,b_-,p_-)^{3/2}\sqrt{\delta}},\\
        \displaystyle \mathfrak{m}_4 &= 3 +\frac{\kappa_4(\widehat{a}_+,b_+, p_+,\widehat{a}_-,b_-,p_-)}{\kappa_2(\widehat{a}_+,b_+, p_+,\widehat{a}_-,b_-,p_-)^2 \delta}
        ,
    \end{aligned}
    \right.
\end{equation}
where $\widehat{a}_+$ and $\widehat{a}_-$ are the estimated parameters already obtained in Section \ref{sec:estimation-a-pm}. In practice, we approach a solution with a least-squares minimization and we denote $(\widehat{b}_+, \widehat{p}_+, \widehat{b}_-, \widehat{p}_-)$ the approximated solution. 

\subsubsection{Risk-neutral parameters} 

Sections \ref{sec:estimation-a-pm} and \ref{sec:estimation-b-p-pm} allow to estimate the parameters of an exp-BGIG process; however these parameters are historical parameters and therefore cannot be used directly for pricing purposes. The last step is to transform them into risk-neutral parameters via the Esscher transform. First, we have determine $\theta^*$ the solution of the Esscher equation \eqref{eq:theta-star}; this can be achieved by any root-finding method on $\R$. Once $\theta^*$ is found, the risk-neutral parameters  $(\widehat{a}_+^*,\widehat{b}_+^*, \widehat{p}_+^*, \widehat{a}_-^*,\widehat{b}_-^*, \widehat{p}_-^*)$ are easily obtained thanks to Proposition \ref{prop:esccher}:
\begin{equation}\label{eq:RN_calibrated}
    \widehat{a}_\pm^* = \widehat{a}_\pm \mp2\theta^*,\quad \widehat{b}_\pm^* =\widehat{b}_\pm\quad\textrm{and}\quad \widehat{p}_\pm^* = \widehat{p}_\pm.
\end{equation}

\subsection{Calibration on S\&P data}
\label{subsec:SP}
We apply the calibration methodology described in Subsection \ref{subsec:calibration_methodo} to daily close prices of the S\&P500 index, for a 2 years trading period (2021-01-01 to 2024-01-01) with $\delta = 1$ (daily returns). To avoid considering outliers, we remove the 1\% highest and lowest log-returns from our dataset. 
The obtained parameters for the daily BGIG distributions are:

\begin{equation}
\label{eq:cal-params}
\left\{
\begin{array}{rclclcrclcl}
    \widehat{a}_+ &=& 5.58753 &\times& 10^2, &\quad& \widehat{a}_- &=& 4.39902 &\times& 10^2, \\
    \widehat{b}_+ &=& 4.43139 &\times& 10^{-2}, &\quad& \widehat{b}_- &=& 2.42973 &\times& 10^{-2}, \\
    \widehat{p}_+ &=& 2.53084 &\times& 10^0, &\quad& \widehat{p}_- &=& 2.26669 &\times& 10^0.
\end{array}
\right.
\end{equation}

In Figure \ref{fig:SP500-fitted}, we display in a same graph the histograms of the log-returns of S\&P 500 index as well as the calibrated density. The corresponding moments and the calibration errors are given in Table \ref{tab:moments}. Let us observe that the kurtosis is strictly bigger than $3$, confirming the presence of fat tails, and that the distribution is skewed to the left, indicating a prevalence of downward moves over upward ones.

\begin{table}
    \begin{center}
    \scriptsize
        \begin{tabular}{ |c||p{2.5cm}|p{2.5cm}||p{2.5cm}|p{2.5cm}|  }
         \hline
         \multicolumn{5}{|c|}{\textbf{Moments Estimation and Errors}} \\
         \hline
         \hline
         \multicolumn{1}{|c||}{}& \multicolumn{2}{c||}{\textbf{Moments}} & \multicolumn{2}{c|}{\textbf{Errors}} \\
         \hline
         \textbf{\textbf{Order}} & \textbf{Empirical} &\textbf{Estimated}& \textbf{Absolute} &\textbf{Percentage} \\
         \hline
         1   & 1.38234e-04 & 1.38194e-04 & 4.02246e-08 & 2.91074e-04\\
         \hline
         2    & 9.29722e-05  &9.29885e-05 & 1.62870e-08 & 1.75151e-04\\
         \hline
         3     & -2.46379e-01 &-2.46271e-01 &1.07698e-04& 4.37315e-04 \\
         \hline
         4    & 3.16266e+00 &3.16231e+00 &3.58132e-04 &1.13250e-04\\
         \hline
        \end{tabular}
    \end{center}
    \caption{Estimated and empirical moments (up to order 4), and the associated errors.}
    \label{tab:moments}
\end{table}

\begin{figure}
        \centering
        \includegraphics[scale = 0.9]{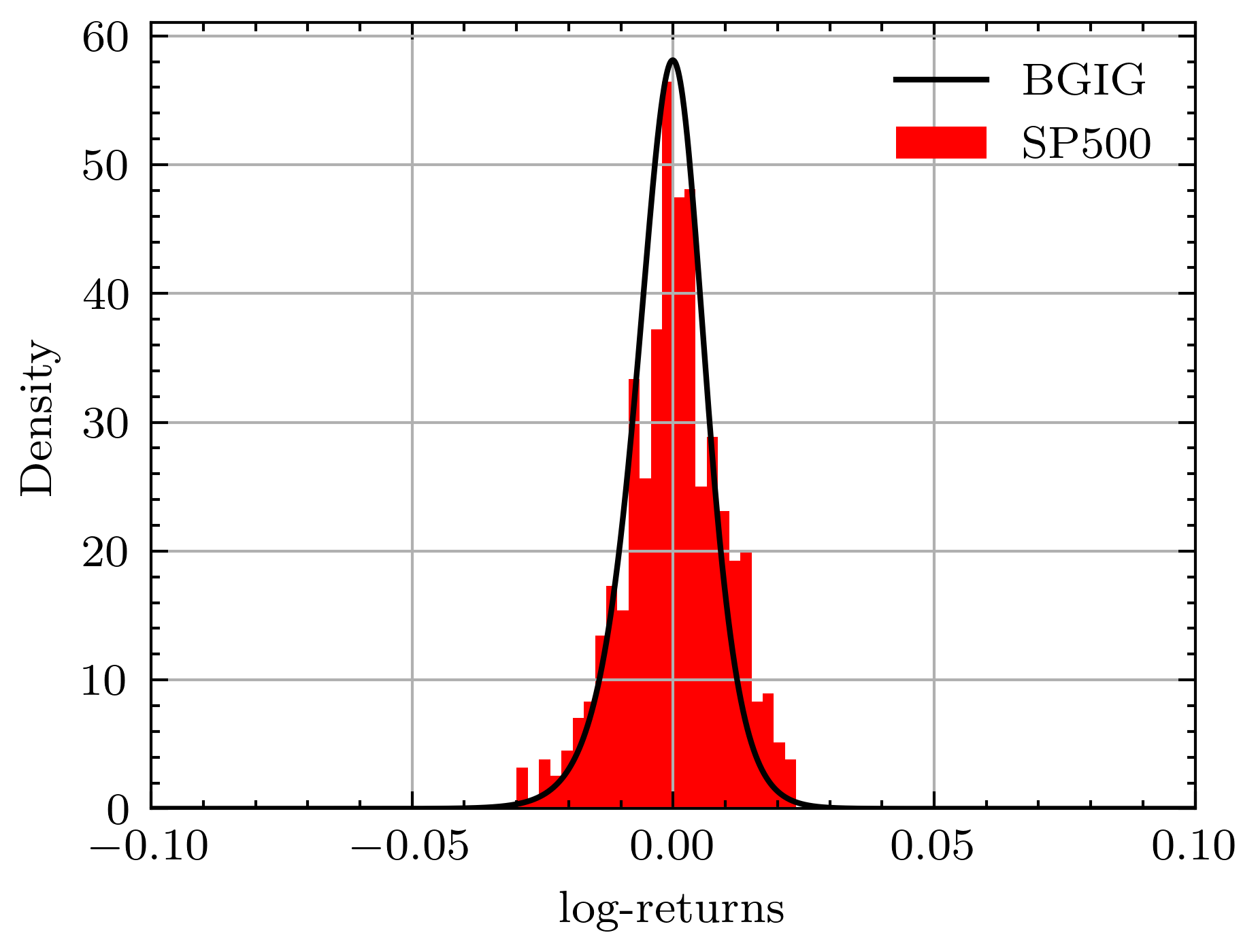}
        \caption{Historical log-returns of the S\&P 500 index and the calibrated BGIG density.
        }
        \label{fig:SP500-fitted}
\end{figure}

\subsection{Option pricing}
\label{subsec:option_pricing}

Let us terminate by pricing some financial instruments using the calibration results from Section \ref{subsec:SP}. We solve the Esscher equation \eqref{eq:theta-star} numerically and we find $\theta^*
    =
    -1.98436 $
and, using \eqref{eq:cal-params} and \eqref{eq:RN_calibrated}, we immediately deduce the risk-neutral parameters. We can then price European call and put options with either Monte-Carlo or Lewis-Lipton \eqref{eq:Lewis} methods. 
Results can be found in Table \ref{tab:prices}.

\begin{table}
    \begin{center}
        \tiny
        \begin{tabular}{ |c||p{1.7cm}|p{1.7cm}||p{1.7cm}|p{1.7cm}|  }
         \hline
         \multicolumn{5}{|c|}{\textbf{European Option Prices}} \\
         \hline
         \hline
         \multicolumn{1}{|c||}{\textbf{Option}}& \multicolumn{2}{c||}{\textbf{Call}} & \multicolumn{2}{c|}{\textbf{Put}} \\
         \hline
         \textbf{$K$} & \textbf{LL} &\textbf{MC}& \textbf{LL} &\textbf{MC} \\
         \hline
         $0.5\times S_0$    & 5.00000e-01 & 5.00098e-01 & 9.25269e-08 &4.00395e-07\\
         \hline
         $0.8\times S_0$    & 2.04505e-01 & 2.04489e-01 & 4.50541e-03 & 4.61485e-03\\
         \hline
         $1\times S_0$      & 6.10931e-02 & 6.10825e-02 &6.10931e-02& 6.10295e-02 \\
         \hline
         $1.2\times S_0$    & 9.52017e-03 & 9.45246e-03 &2.09520e-01 &2.08716e-01\\
         \hline
         $1.5\times S_0$    & 2.25386e-04 & 2.40956e-04& 5.00225e-01& 5.00270e-01\\
         \hline
        \end{tabular}
    \end{center}
    \caption{Prices of European calls/puts with $S_0=1$, $T=252$ days, $r=0\%$, and $n=50~000$ paths for the MC method. 
    MC prices are more accurate for put options since the payoff is bounded.
    }
    \label{tab:prices}
\end{table}

\section{Conclusion}
\label{sec:conclusion}

We have introduced and studied BGIG distributions and processes. These distributions, constructed by convolution of GIG random variables, are proved to be smooth, real-analytic on $\R\backslash\{0\}$ and unimodal. Since they are also infinitely divisible, they generate associated L\'evy processes, that we have shown have finite variations, infinite activity and a Blumenthal-Getoor index of $1/2$. With the help of known algorithms, these processes can also be efficiently simulated. Moreover they are stable under Esscher transform, which makes them suitable for financial 
modeling and option pricing.

As an application of this theory, we have defined a stock market model based on exponential BGIG processes and provided a calibration methodology relying on the asymptotics for extrema of BGIG distributions and on the method of moments; we have tested this methodology over historical stock market returns, and,  combined with the estimation of the Esscher parameter, we have applied the model to the pricing of S\&P 500 index options via Monte-Carlo and Fourier based methods. 


Future work should include calibration of the model over option data, investigation of the BGIG implied volatility smile and comparison with the market smile, and pricing of various generations of exotic options using state-of-the-art pricing techniques. The existence of  a (semi-)closed pricing formula for European-style options, e.g. via series expansions, would also be an interesting question to address.

\section{Acknowledgement}

We thank Iosif Pinelis for his help proving Proposition \ref{prop:max-min-asymptotic}.

\bibliographystyle{plain}
\bibliography{biblio_stoch_proc}






\end{document}